\documentclass[11pt]{amsart}
\usepackage{amssymb}
\usepackage{amscd}
\usepackage[latin1]{inputenc}
\usepackage{amsmath}
\usepackage{amsfonts}
\usepackage{latexsym}
\usepackage{verbatim}
\usepackage{graphicx}
\usepackage{epsfig}
\usepackage{xcolor}
\usepackage{enumitem}

\newtheorem{theorem}{Theorem}[section]
\newtheorem{corollary}[theorem]{Corollary}
\newtheorem{lemma}[theorem]{Lemma}
\newtheorem{question}[theorem]{Question}
\newtheorem{proposition}[theorem]{Proposition}

\theoremstyle{definition}
\newtheorem{definition}[theorem]{Definition}

\newtheorem{remark}[theorem]{Remark}

\numberwithin{equation}{section}

\newcommand{\ppi}{\boldsymbol{\pi}}
\newcommand{\ww}{\mathbf{w}}
\newcommand{\zz}{\mathbb{Z}}
\newcommand{\LLL}{\mathcal{L}}
\newcommand{\LL}{\mathcal{L}}
\newcommand{\GGG}{\mathcal{G}}
\newcommand{\CCC}{\mathcal{C}}
\newcommand{\MMM}{\mathcal{M}}

\newcommand{\FFF}{\mathcal{F}}
\newcommand{\BBB}{\mathcal{B}}
\newcommand{\DDD}{\mathcal{D}}
\newcommand{\NN}{\mathbb{N}}
\newcommand{\ph}{\varphi}
\newcommand{\RR}{\mathbb{R}}
\newcommand{\ZZ}{\mathbb{Z}}
\newcommand{\abs}[1]{\left\lvert#1\right\rvert}
\newcommand{\defn}[1]{\textbf{#1}}

\DeclareMathOperator{\Per}{Per}

\begin{document}

\title[One-sided almost specification and intrinsic ergodicity]{One-sided almost specification and intrinsic ergodicity}

\begin{abstract}

%Shift spaces with the specification property are intrinsically ergodic, i.e.\ they have a unique measure of maximal entropy.  This can fail for shifts with the weaker \emph{almost specification} property. 
We define a new property called \emph{one-sided almost specification}, which lies between the properties of specification and almost specification, and prove that it guarantees intrinsic ergodicity (i.e. uniqueness of the measure of maximal entropy) if the corresponding mistake function $g$ is bounded.  We also show that uniqueness may fail for unbounded $g$ such as $\log\log n$. Our results have consequences for almost specification: we prove that almost specification with $g \equiv 1$ implies one-sided almost specification (with $g \equiv 1$), and hence uniqueness. On the other hand, the second author showed recently that almost specification with $g \equiv 4$ does not imply uniqueness. %This leaves open the question of whether almost specification implies intrinsic ergodicity when $g=2$ or $g=3$.

\end{abstract}

\date{\today}
\author{Vaughn Climenhaga}
\address{Vaughn Climenhaga\\
Department of Mathematics\\
University of Houston\\
4800 Calhoun St.\\
Houston, TX 77204}
\email{climenha@math.uh.edu}
\urladdr{http://www.math.uh.edu/$\sim$climenha/}
\thanks{}
\author{Ronnie Pavlov}
\address{Ronnie Pavlov\\
Department of Mathematics\\
University of Denver\\
2280 S. Vine St.\\
Denver, CO 80208}
\email{rpavlov@du.edu}
\urladdr{www.math.du.edu/$\sim$rpavlov/}
\thanks{V.C.\ is partially supported by NSF grant DMS-1362838.  R.P.\ is partially supported by NSF grant DMS-1500685.}
\renewcommand{\subjclassname}{MSC 2010}
%below are the definitions for some subject classification numbers
%22D40 Ergodic theory on groups 
%37A05 Measure-preserving transformations 
%37A15 General groups of measure preserving transformations 
%37A35 Entropy and other invariants, isomorphism, classification 
%37B10 symbolic dynamics
%37B40 topological entropy
%37B50 multi-dimensional shifts of finite type, tiling systems 
%37C40 smooth ergodic theory, invariant measures
%37C45 dimension theory of dynamical systems
%37C85 Dynamics of group actions other than Z and R, and foliations 
%37C99 smooth dynamical systems, general theory
%37D35 thermodynamic formalism, variational principles, equilibrium states
\maketitle

\section{Introduction}

%\vc{I edited the introduction to avoid plagiarizing your other paper too much.  I also added the statement of the result on products and factors to the introduction, I think it belongs here.  Finally, I formulated the  open questions a little more formally, since I think they deserve highlighting.}
The notion of ``entropy'' in dynamics can be defined as an invariant of a map preserving a probability measure (measure-theoretic entropy) and of a continuous map on a compact metric space (topological entropy).  These are related by the variational principle, which states that the topological entropy of a topological dynamical system $(X,T)$ is the supremum of the measure-theoretic entropies taken over all probability measures preserved by $T$.  A measure on $X$ that achieves this supremum is called a \textbf{measure of maximal entropy (MME)}.  If there is a unique MME, the system $(X,T)$ is said to be \textbf{intrinsically ergodic} \cite{bW70}.

Existence and uniqueness of MMEs is a central question in thermodynamic formalism, which relates ergodic theory and topological dynamics.  It is often the case that the unique MME for an intrinsically ergodic system has strong statistical properties, and related thermodynamic considerations (equilibrium states for non-zero potentials) are connected to properties of `physical measures' for smooth systems.  

We study intrinsic ergodicity in the context of symbolic dynamics over a finite alphabet, where existence of an MME is automatic, and so the real question is uniqueness.  Intrinsic ergodicity for mixing subshifts of finite type was proved by Parry \cite{wP64} using Perron--Frobenius theory; a different proof for subshifts with the \textbf{specification property} (which includes the class of mixing SFTs) was given by Bowen \cite{rB74}.  This property requires that any sequence of orbit segments can be shadowed by a single orbit that transitions from each segment to the next in uniformly bounded time; in symbolic dynamics, orbit segments correspond to \textbf{words} in the language of the shift.

Recently a number of weakened versions of the specification property have been used to study various questions in ergodic theory.  This includes \textbf{almost specification} \cite{kwietniaketal,pavlovAS,pfister-sullivan2,pfister-sullivan,thompson}, which allows one to concatenate arbitrary words in the language into a new word in the language if a ``small'' number of letters are allowed to change in each word.  The number of changes is controlled by a sublinear ``mistake function'' $g(n)$; see Section~\ref{defs} for a formal definition.  Almost specification is enough to establish various results in large deviation theory \cite{pfister-sullivan} and multifractal analysis \cite{pfister-sullivan2,thompson}, but it was unknown for some time whether almost specification also implies intrinsic ergodicity.  This question was answered in the negative by (independently) \cite{KOR} and \cite{pavlovAS}; in fact, \cite[Theorem 1.2]{pavlovAS} shows that intrinsic ergodicity may fail even with the constant mistake function $g \equiv 4$. (Here and elsewhere, $g \equiv C$ means that $g$ is the constant function $C$.)

One motivation for almost specification is the fact that many natural examples satisfy it for small $g$, such as the classical 
$\beta$-shifts, which have almost specification with $g \equiv 1$ and do satisfy intrinsic ergodicity. In fact, $\beta$-shifts satisfy a slightly stronger property; when one wishes to concatenate some words $w^{(1)}$, $\ldots$, $w^{(n)}$ together, it suffices to make the permitted number of changes to the words $w^{(1)},\dots, w^{(n-1)}$, and leave the final word $w^{(n)}$ untouched.   Though this might seem like an incremental strengthening, it proves quite important. We call this stronger notion \textbf{left almost specification (LAS)}; our main result is that this property actually does imply intrinsic ergodicity if the mistake function $g$ is bounded.

\begin{theorem}\label{thm:las-main}
If $X$ is a subshift with left almost specification for a bounded mistake function, then it has a unique measure of maximal entropy $\mu$.  Moreover, $\mu$ is the limiting distribution of periodic orbits; finally, the system $(X,\sigma,\mu)$ is Bernoulli, and has exponential decay of correlations and the central limit theorem for H\"older observables.
\end{theorem}

We prove Theorem \ref{thm:las-main} in \S\ref{sec:main-pf}.  This theorem covers the case when $X$ has the usual specification property (see Lemma \ref{LASspec}).  We show in \S\ref{sec:corollary-pf} that it also covers the case when $X$ has the usual almost specification property with $g \equiv 1$, and deduce the following.

\begin{corollary}\label{thm:as-main}
If $X$ is a subshift with almost specification for the mistake function $g \equiv 1$, then it has a unique measure of maximal entropy $\mu$, and this measure satisfies the conclusions of Theorem~\ref{thm:las-main}.
\end{corollary}

Together with the result from \cite{pavlovAS} that $g \equiv 4$ is not enough to guarantee uniqueness, we are led to the following open problem.

\begin{question}\label{q:23}
Does almost specification with mistake function $g \equiv 2$ or $g \equiv 3$ guarantee intrinsic ergodicity?
\end{question}

We observe that the approach of Corollary \ref{thm:as-main} does not immediately address Question \ref{q:23}, %go through, 
since once we allow $g > 1$ we can no longer deduce one-sided almost specification.

Returning to shifts with LAS, we observe that the requirement of bounded mistake function cannot be relaxed too far; we prove that even doubly logarithmic growth of $g$ does not imply intrinsic ergodicity. % for left almost specification. 

\begin{theorem}\label{doublelogex}
There is a subshift $X$ with LAS with $g(n) = 1 + 2 \lfloor \log_2 \log_2 n \rfloor$ and multiple measures of maximal entropy (with disjoint supports).
\end{theorem}

This immediately leads to the following open problem.

\begin{question}\label{q:unbdd}
Is there any unbounded function $g(n)$ such that LAS with mistake function $g$ implies intrinsic ergodicity?
\end{question}

We omit details here, but using the example of Theorem \ref{doublelogex} as a starting point, one can use higher-power presentations just as in \cite{pavlovAS} to create examples satisfying LAS with $g(n) < C \log \log n$ and multiple MMEs for arbitrarily small $C$. So, the mistake functions $g$ for which we do not know whether LAS implies uniqueness of the MME are those which are unbounded, but for which $\frac{g(n)}{\log \log n} \rightarrow 0$ as $n \rightarrow \infty$.  See Remark \ref{rmk:unbdd} and \S\ref{sec:unbdd} for the point at which our proof breaks down when $g$ is unbounded.

The example %with multiple MMEs 
from Theorem~\ref{doublelogex} %which we will construct in \S\ref{sec:loglog-pf} 
has exactly two ergodic MMEs, whose supports are disjoint. It is actually possible for a subshift to have two ergodic MMEs that are \emph{both} fully supported -- this is the case, for example, with the Dyck shift. This leads to the following natural question.

\begin{question}\label{q:fully-supported}
Is there a subshift $X$ with LAS for some sublinear mistake function for which there are multiple fully supported ergodic MMEs?
\end{question}

The mechanism for uniqueness in Theorem \ref{thm:las-main} is a result from \cite{climenhaga1} that uses a certain non-uniform specification condition (superficially quite different from almost specification) to model $X$ with a countable state Markov shift. %A key intermediate step is the proof of a property called \emph{entropy minimality} for the subshift $X$ (see Definition~\ref{entmindef}). The following result shows that there is no unbounded $g$ for which irreducibility and LAS with mistake function $g$ implies entropy minimality. This means that if any such $g$ implies intrinsic ergodicity, it could not be proved via our methods.
%\begin{theorem}\label{notentmin}
%For every unbounded $g(n)$, there is an irreducible subshift $X$ satisfying LAS with mistake function $g$ which is not entropy minimal.
%\end{theorem}
The conditions in \cite{climenhaga1} are mild variants of conditions introduced in \cite{CT}, where they were used to prove that not only are $\beta$-shifts intrinsically ergodic, but so are their subshift factors.  Although the conditions from \cite{CT} are well-behaved under passing to factors, their behaviour upon passing to products is not immediately clear.  On the other hand, as we will show, the hypotheses of Theorem \ref{thm:las-main} are stable under both products and factors. %\vc{Should we elaborate on this by describing in more detail some new examples?  Intrinsic ergodicity of products of $\beta$-shifts follows already from Buzzi's 1997 paper in Monatshefte; do we have any other new examples we can cover easily?}

\begin{proposition}\label{prodfact}
If $X,Y$ are subshifts satisfying LAS with bounded $g$, then so is their product $X\times Y$.  Similarly, if $X$ is a subshift satisfying LAS with bounded $g$, and $Z$ is a symbolic factor of $X$, then $Z$ satisfies LAS with bounded $g$.
%Left almost specification with bounded $g$ is preserved under products and (symbolic) factors.
\end{proposition}

\begin{remark}
Since $\beta$-shifts satisfy LAS with $g \equiv 1$, Proposition~\ref{prodfact} implies that any shift created by taking iterated products/factors of $\beta$-shifts also has a unique MME. For example, given any $\beta, \beta'$, we can define $X_{\beta, \beta'}$ to be the subshift consisting of all sequences created by coordinate-wise sums of points from the $\beta$-shifts $X_{\beta}$ and $X_{\beta'}$. It is clear that $X_{\beta, \beta'}$ is a factor of the direct product $X_{\beta} \times X_{\beta'}$ (induced by the letter-to-letter map $(i,j) \mapsto i + j$), and so every $X_{\beta, \beta'}$ also has a unique MME. It is quite possible that the methods of \cite{CT} or \cite{climenhaga1} could also be used to prove uniqueness of MME for these shifts, but it would be more difficult.
\end{remark}

In \S\ref{defs} we give all the relevant definitions.  Theorem \ref{thm:las-main} is proved in \S\ref{sec:main-pf}, and the remainder of the proofs are given in \S\ref{sec:other-pfs}.
%; Corollary \ref{thm:as-main}, Theorem \ref{doublelogex}, and Proposition \ref{prodfact} are proved in \S\ref{sec:other-pfs}.  
%In \S\ref{sec:open} we discuss Questions~\ref{q:23}, \ref{q:unbdd}, and \ref{q:fully-supported} a little further.

\subsection*{Acknowledgments}
We are grateful to the anonymous referee for a careful reading and many helpful comments, including suggestions that led to the current proofs (and formulations) of Theorems \ref{measure-center} and \ref{measure-center-2}.

\section{Definitions and preliminaries}\label{defs}

\subsection{Symbolic dynamics}

%\vc{I reworded a few things to avoid plagiarization of \cite{pavlovAS}, and combined a couple consecutive definitions to make things flow}
\begin{definition}
Let $A$ be a finite set, which we call an \textbf{alphabet}.  The \textbf{full shift} over $A$ is the set $A^{\zz} = \{\cdots x(-1) \ x(0) \  x(1) \cdots \ : \ x(i) \in A\}$, which is a compact %topological space with the (discrete) product topology.
metric space with the metric $\rho$ defined by $\rho(x,y) = 2^{-\min \{ |k| : x(k) \neq y(k) \}}$ for $x \neq y$.  

%\vc{I find the definition of words in terms of $A^{[i,j]}$ (instead of just $A^n$) confusing; the convention that I've always seen is that a word just records a sequence of symbols, without specifying what index they start at.  Of course a cylinder needs a specified starting location, but for the stuff we're doing I think words suffice.}
A \textbf{word} over $A$ is an element of $A^* := \bigcup_{n=0}^\infty A^n$; here $A^0 = \{\varepsilon\}$, where $\varepsilon$ is the so-called empty word. Given a word $w\in A^n$, we write $|w|=n$ for the \textbf{length} of the word. (We will also use $|S|$ to refer to cardinalities of finite sets, however the usage should always be clear from context.) Given $x\in A^\ZZ$ and $i < j \in \ZZ$, we write $x([i,j]) = x(i) x(i+1) \cdots x(j)$ for the word of length $j-i+1$ that begins at position $i$ and ends at position $j$.
\end{definition}

%\begin{definition}
%A \textbf{word} over $A$ is a member of $A^{[i,j]}$ for some $i<j$, whose \textbf{length} $j-i+1$ is denoted by $|w|$. The \textbf{length} of a word $w$ is simply the number of symbols in it, and is denoted by $|w|$. The set $\bigcup_{i,j \in \zz, i<j} A^{[i,j]}$ of all words over $A$ is denoted by $A^*$. For any $n$, we use $A^n$ to denote the set $A^{[1,n]}$. (Here and elsewhere, for $i<j \in \mathbb{Z}$, we use $[i,j]$ to denote $[i,j] \cap \mathbb{Z}$.)
%\end{definition}

\begin{definition}
The \textbf{shift action}, denoted by $\{\sigma^n\}_{n \in \zz}$, is the $\zz$-action on a full shift $A^{\zz}$ defined by $(\sigma^n x)(m) = x(m+n)$ for $m,n \in \zz$. 
%\end{definition}
%\begin{definition}
A \textbf{subshift over $A$} is a closed subset of the full shift $A^{\zz}$ that is invariant under the shift action.  Every subshift is a compact metric space with respect to the induced metric from $A^\zz$.
%, which is a compact space with the induced topology from $A^{\zz}$.
\end{definition}

The shift $\sigma := \sigma^1$ is a homeomorphism on any subshift $X$, making $(X,\sigma)$ a topological dynamical system. A  subshift is characterized by its list of forbidden words $\mathcal{F} \subset A^*$: given any such $\mathcal{F}$, the set $X_{\mathcal{F}} := \{x \in A^{\zz} \ : \ x([i,j]) \notin \mathcal{F} \ \forall i,j \in \zz, i<j\}$ is closed and shift-invariant, and any subshift can be represented in this way.%\vc{Changed notation to $X_\FFF$ to avoid conflict with $X(\DDD)$ below.}

\begin{definition}
The \textbf{language} of a subshift $X$, denoted by $\mathcal{L}(X)$, is the set of all words which appear in points of $X$. For any $n \in \NN$, we write $\mathcal{L}_n(X) := 
\mathcal{L}(X) \cap A^n$ for the set of words in the language of $X$ with length $n$.  
We will also need to deal with collections of words $\DDD \subset \LLL(X)$, and given such a collection we write $\DDD_n = \DDD \cap \LLL_n(X)$ for the set of words in $\DDD$ with length $n$.%\vc{There was a comment in the original version about only using words from $A^n$ instead of $A^{\{i,\dots,j\}}$, and being able to consider two words as the same if they are shifts of each other.  It seems simpler to me to just dodge this issue entirely by making the definition via $A^n$, since the other version doesn't seem necessary for us.} 
\end{definition}

%In the previous definition, we dealt only with words from $A^n$ rather than $A^{\{i,\ldots,j\}}$ for arbitrary $i < j$; this is because any word in $A^{\{i,\ldots,j\}}$ can clearly be thought of as a word in $A^{j-i+1}$ by simply shifting it. We will generally consider two words to be the same if they are shifts of each other. 

A collection of words $\DDD \subset A^*$ is said to be \textbf{factorial} if it is closed under passing to subwords.  The language $\LLL(X)$ of a subshift is factorial, but we will also have to deal with subsets $\DDD \subset \LLL(X)$ that are not factorial. Observe that if $\DDD$ is factorial, then
\begin{equation}\label{eqn:submult}
|\DDD_{m+n}| \leq |\DDD_m| \cdot |\DDD_n| \text{ for all }m,n\in \NN.
\end{equation}

\begin{definition}
For any subshift $X$ and word $w \in \mathcal{L}_n(X)$, the \textbf{cylinder set} $[w]$ is the set of all $x \in X$ with $x([1,n]) = w$.  
%\end{definition}
%\begin{definition}
The shift $X$ is \textbf{irreducible} (or \textbf{topologically transitive}) if for any $u,v \in \mathcal{L}(X)$, there exists $n \in \mathbb{N}$ so that 
$[u] \cap \sigma^{-n} [v] \neq \varnothing$, or equivalently if there exists a word $w$ so that $uwv \in \mathcal{L}(X)$.
\end{definition}

\subsection{Thermodynamic formalism for subshifts}

We recall some of the main elements of thermodynamic formalism as they appear in the context of subshifts; further details and results can be found in \cite{walters}.

\begin{definition}\label{topent}
The \textbf{topological entropy} of a subshift $X$ is
\begin{equation}\label{eqn:hX}
h(X) := \lim_{n \rightarrow \infty} \frac{1}{n} \ln |\mathcal{L}_n(X)| 
= \inf_{n\in \NN} \frac 1n \ln|\mathcal{L}_n(X)|,
\end{equation}
where existence of the limit and the second equality follow from \eqref{eqn:submult} and a standard lemma regarding subadditive sequences.
We will also need to consider the entropy of a subset of the language: given $\mathcal{D} \subset \mathcal{L}(X)$, we write
\begin{equation}\label{eqn:hD}
h(\mathcal{D}) := \limsup_{n\to \infty} \frac 1n \ln |\mathcal{D}_n|,
\end{equation}\textsl{}
where in general the limit may not exist.
\end{definition}

We will make occasional use of the following elementary fact:
\begin{equation}\label{eqn:h-union}
h(\mathcal{C} \cup \mathcal{D}) = \max(h(\mathcal{C}),h(\mathcal{D})).
\end{equation}

%\vc{Moved this comment and lemma to here since it doesn't involve measures (it was originally after the definition of measure-theoretic entropy). I also reformulated it to include factorial collections $\DDD$.}

If $\DDD \subset A^*$ is factorial, then just as with $\LLL(X)$, \eqref{eqn:submult} shows that the limit in \eqref{eqn:hD} exists and is equal to $\inf_n \frac 1n \ln |\DDD_n|$.  The following consequence of this is useful enough to be worth stating formally.
%In Definition~\ref{topent}, a standard subadditivity argument shows that the limit can be replaced by an infimum; i.e. for any $n$, $h(X) \leq \frac{1}{n} \ln |\mathcal{L}_n(X)|$. This implies the following fact.

\begin{lemma}\label{wordcount}
For any factorial $\DDD \subset A^*$, we have
$
|\DDD_n| \geq e^{nh(\DDD)}
$
for every $n\in \NN$.
In particular, for any subshift $X$, we have
$
|\mathcal{L}_n(X)| \geq e^{nh(X)}.
$
\end{lemma}

Given a factorial set of words $\DDD\subset A^*$, one can define a subshift $X(\DDD)$ by the condition that $x\in X(\DDD)$ if and only if every subword of $x$ is in $\DDD$. The language of $X(\DDD)$ is contained in $\DDD$, so clearly 
$h(X(\DDD)) \leq h(\DDD)$. However, we may have $\LLL(X(\DDD)) \subsetneq \DDD$, since there could be words in $\DDD$ that do not appear as subwords of arbitrarily long words of $\DDD$. Accordingly, we will need the following result, which is proved in \S\ref{sec:entropy-pf} and applies even when $\LLL(X(\DDD))$ is smaller than $\DDD$.

%\vc{I moved the proof of Lemma \ref{entbigger} to \S\ref{sec:other-pfs}, I think it goes better there, rather than disrupting the flow of ideas here.  I also reformulated this as an equality, since the other inequality is immediate.}

\begin{lemma}\label{entbigger}
Let $\DDD$ be a factorial set of words.  Then 
 $h(X(\mathcal{D})) = h(\mathcal{D})$.
\end{lemma}

%\vc{Defined notation $\mathcal{M}$ here since we use it later}
Now we recall some  definitions from measure-theoretic dynamics.  We write $\mathcal{M}(A^\ZZ)$ for the space of $\sigma$-invariant Borel probability measures on the full shift, and similarly $\mathcal{M}(X)$ will denote the elements of $\mathcal{M}(A^\ZZ)$ that give full weight to $X$.

\begin{definition}\label{measent}
The \textbf{measure-theoretic entropy} of $\mu\in \mathcal{M}(A^{\zz})$ is
\[
h(\mu) := \lim_{n \rightarrow \infty} \frac{-1}{n}  \sum_{w \in A^n} \mu([w]) \ln \mu([w]),
\]
where terms with $\mu([w]) = 0$ are omitted from the sum.  (Existence of the limit is a standard result and again uses subadditivity.)
\end{definition}

The \textbf{variational principle} \cite[Theorem 8.6]{walters} relates the two kinds of entropy: for every subshift $X$, we have $h(X) = \sup \{h(\mu) : \mu \in \MMM(X)\}$.

\begin{definition}
A \textbf{measure of maximal entropy} (MME) for a subshift $X$ is an invariant measure $\mu\in \mathcal{M}(X)$ for which $h(\mu) = h(X)$.  The shift $X$ is said to be \textbf{intrinsically ergodic} if it has a unique MME.
\end{definition}

%The classical variational principle says that for any subshift $X$, $\sup_{\mu} h(\mu) = h(X)$, where the supremum, taken over all shift-invariant Borel probability measures whose support is contained in $X$, is achieved. Therefore, any subshift has at least one measure of maximal entropy. 

%We will need to make use of the following well-known fact about the full shift.

%\begin{lemma}\label{fsunique}
%For any alphabet $A$, the full shift $A^{\mathbb{Z}}$ has a unique measure of maximal entropy, namely the measure $\mu$ with $\mu([w]) = |A|^{-n}$ for all $n$ and $w \in A^n$.
%\end{lemma}

We will need the following standard result on construction of measures with large (or full) entropy, which follows from the second half of the proof of \cite[Theorem 8.6]{walters}.
%\vc{Added this lemma since we need it in two places, so it seemed to make sense to have a formal statement}

\begin{lemma}\label{lem:build-mme}
Let $X$ be a subshift and $n_k\to\infty$.  Given $\DDD_{n_k} \subset \LLL_{n_k}(X)$ and any choice of $x_v\in [v]$ for each $v\in \DDD_{n_k}$, consider the measures $\nu_k = \frac 1{|\DDD_{n_k}|} \sum_{v\in \DDD_{n_k}} \delta_{x_v}$ and $\mu_k = \frac 1{n_k} \sum_{i=0}^{n_k-1} \sigma_*^i \nu_k$.  Then any weak* limit point $\mu$ of the sequence $\mu_k$ is a $\sigma$-invariant measure with $h(\mu) \geq \liminf \frac 1{n_k} \ln |\DDD_{n_k}|$.  In particular, if $\liminf \frac 1{n_k} \ln |D_{n_k}| = h(X)$, then any weak* limit point of the sequence $\mu_k$ is an MME for $X$.
\end{lemma}

Our main goal is to prove that certain subshifts are intrinsically ergodic. A key intermediate step in our approach will be to establish the property of \textbf{entropy minimality}, first defined in \cite{CS}. 

\begin{definition}\label{entmindef}
A subshift $X$ is \textbf{entropy minimal} if every nonempty proper subshift $X' \subsetneq X$ has topological entropy less than $h(X)$.
\end{definition}

%We note that the following definition is clearly equivalent: 
Two equivalent formulations are easily checked and will be useful later. Firstly, $X$ is clearly entropy minimal iff for all $w \in \mathcal{L}(X)$, the subshift
\[
X_w := \{x \in X \ : \ x \textrm{ does not contain $w$ as a subword}\}
\]
has topological entropy less than $h(X)$. Also, $X$ is entropy minimal iff all measures of maximal entropy of $X$ are \textbf{fully supported} (meaning that every open set in $X$ has positive measure). In particular, entropy minimality holds whenever $X$ is intrinsically ergodic and the unique MME is fully supported. Such shifts include all mixing SFTs and, more generally, the class of shifts that we will study.%\vc{Where does entropy minimality first appear in the literature? We should give a reference.} \changed{Another equivalent formulation is that every MME for $X$ is fully supported.}

\subsection{Specification properties}

By upper semi-continuity of the entropy function $h\colon \mathcal{M}(X) \to [0,\ln |A|]$ (or by applying Lemma \ref{lem:build-mme} with $\DDD = \LLL(X)$), every subshift $X$ has at least one MME.  We will be interested in the question of uniqueness, taking the following definition as our starting point.

\begin{definition}
A subshift $X$ has the \textbf{specification} property if  there is $\tau\in \NN$ such that for every $v,w\in \LLL(X)$ there is $u\in \LLL(X)$ such that $|u|= \tau$ and $vuw\in \LLL(X)$.
\end{definition}

It was shown by Bowen \cite{bowen} that specification implies uniqueness of the MME.\footnote{Bowen's original definition is made in a more general setting and has a few more hypotheses, including periodicity of the ``glued'' word; in the setting of subshifts it turns out that our definition is equivalent to his \cite{bertrand}.}  More recently, various weakenings of the specification property have been introduced.  The version we will focus on was introduced by Pfister and Sullivan \cite{pfister-sullivan} as the ``$g$-almost product property''; we will follow the convention of \cite{thompson} and call this property \textbf{almost specification}. 

In this definition and the remainder of the paper, the distance between two words $v,w$ (which is only defined when $|v| = |w|$) is always assumed to be the \textbf{Hamming distance} $d_H(v,w) := |\{i \ : \ v(i) \neq w(i)\}|$.  We will frequently use the associated \textbf{Hamming balls} (of radius $m$)%\vc{$v\in \LLL(X)$, or $v\in A^*$?}
\begin{equation}\label{eqn:ham-ball}
B_m(w) := \{v\in \LLL_{|w|}(X) \ : \ d_H(v,w) \leq m\}.
\end{equation}

\begin{definition}\label{ASdef}
A subshift $X$ has \textbf{almost specification (or AS) with mistake function $g(n)$} if%\vc{Would it be easier notationally to write $w_i$ instead of $w^{(i)}$ for sequences of words, since we use $w(i)$ instead of $w_i$ for the letters of a word?  Right now we use $w^{(i)}$ in some places and $w_i$ in others, it's not completely consistent.}
\begin{itemize}
\item
$\frac{g(n)}{n} \rightarrow 0$
\item For any words $w^{(1)}$, $w^{(2)}$, $\ldots$, $w^{(k)} \in \mathcal{L}(X)$, there exist words $\bar{w}^{(1)}$, $\bar{w}^{(2)}$, $\ldots$, $\bar{w}^{(k)} \in \mathcal{L}(X)$ so that $d_H(\bar{w}^{(i)}, w^{(i)}) \leq g(|w^{(i)}|)$ for
$1 \leq i \leq k$, and the concatenation $\bar{w}^{(1)} \bar{w}^{(2)} \ldots \bar{w}^{(k)}$ is in $\mathcal{L}(X)$.
\end{itemize}
\end{definition}

We will also consider the following slightly stronger property.%\vc{I'm not sure how I feel about the abbreviation ``AS'' -- in my head it automatically expands to ``almost surely''.  Probably an effect of spending too much time around probabilists...}

\begin{definition}\label{LASdef}
A subshift $X$ has \textbf{left almost specification (or LAS)  with mistake function $g(n)$} if
\begin{itemize}
\item $\frac{g(n)}{n} \rightarrow 0$
\item For any words $w^{(1)}, w^{(2)} \in \mathcal{L}(X)$, there exists a word $\bar{w}^{(1)}$ in $\mathcal{L}(X)$ where $d_H(\bar{w}^{(1)}, w^{(1)}) \leq g(|w^{(1)}|)$ and $\bar{w}^{(1)} w^{(2)} \in \mathcal{L}(X)$.
\end{itemize}
\end{definition}

We first quickly demonstrate that specification does in fact imply LAS with bounded $g$.%\vc{I'm not completely convinced we need to prove this formally, I'd be happy to just state it as `easy to prove' and then move on.  I'm also happy to leave it as is, I don't object to it being here.}

\begin{lemma}\label{LASspec}
If $X$ has specification, then it has LAS with bounded $g$.
\end{lemma}
\begin{proof}
Suppose that $X$ has specification; then there exists $\tau$ so that for all $v,w \in \mathcal{L}(X)$, there exists $u \in \mathcal{L}(X)$ with $|u| = \tau$ such that $vuw \in \mathcal{L}(X)$. We claim that $X$ has LAS with $g \equiv \tau$. To see this, choose any $w^{(1)}, w^{(2)} \in \mathcal{L}(X)$.
If $|w^{(1)}| \leq \tau$, then choose any word $\bar{w}^{(1)}$ with length $|w^{(1)}|$ which can precede $w^{(2)}$ in a point of $X$; then trivially $d_H(w^{(1)}, \bar{w}^{(1)}) \leq \tau$ and $\bar{w}^{(1)} w^{(2)} \in \mathcal{L}(X)$. If $|w^{(1)}| > \tau$, then define $v$ to be the word obtained by removing the final $\tau$ letters of $w^{(1)}$. Then by specification, there exists $u$ with $|u| = \tau$ so that $vu w^{(2)} \in \mathcal{L}(X)$. Again $d_H(vu, w^{(1)}) \leq \tau$. In both cases, we found $\bar{w}^{(1)}$ with $d_H(w^{(1)}, \bar{w}^{(1)}) \leq \tau$ and $\bar{w}^{(1)} w^{(2)} \in \mathcal{L}(X)$, so $X$ has LAS with $g \equiv \tau$.
\end{proof}

The reader may check that if $X$ has LAS with $g(n)$, it also has LAS with $g'(n)$ defined by 
$g'(n) = \min\{g(k): k\geq n\}$, and so we always assume without loss of generality that $g$ is nondecreasing.
There are many subshifts known to satisfy AS and/or LAS; for instance, any $\beta$-shift has LAS\ with gap function $g \equiv 1$ (see \cite{pfister-sullivan2}); since $\beta$-shifts only have specification for a Lebesgue-null set of values of $\beta$ (see \cite{buzzi,schmeling}), this also demonstrates that the converse of Lemma \ref{LASspec} fails, and so LAS with bounded $g$ is a more general property than specification. Many of the so-called $S$-gap shifts satisfy AS\ (with gap function dependent on $S$).  Obviously there is a corresponding notion of right almost specification, and all of our proofs carry over to that case in a standard way. (See the proof of Corollary~\ref{thm:as-main} for more details.) So, though our results are stated for subshifts with LAS, they really apply to subshifts satisfying either type of ``one-sided'' almost specification.

We observe that LAS implies AS.

\begin{lemma}\label{LAS=>AS}
If a subshift has LAS\ with mistake function $g(n)$, then it has AS\ with mistake function $g(n)$.
\end{lemma}
\begin{proof}
We claim that LAS\ with mistake function $g(n)$ implies the following statement, which implies AS\ with mistake function $g(n)$: for any words $w^{(1)}$, $w^{(2)}$, $\ldots$, $w^{(k)} \in \mathcal{L}(X)$, there exist words $\bar{w}^{(1)}$, $\bar{w}^{(2)}$, $\ldots$, $\bar{w}^{(k-1)} \in \mathcal{L}(X)$ so that $d_H(\bar{w}^{(i)}, w^{(i)}) \leq g(|w^{(i)}|)$ for $1 \leq i \leq k-1$, and  $\bar{w}^{(1)} \bar{w}^{(2)} \ldots \bar{w}^{(k-1)} w^{(k)}$ is in $\mathcal{L}(X)$. (Here, note that $\bar{w}^{(k)}$ was not defined; $w^{(k)}$ does not need to be changed in the concatenation.)

This is proved via induction on $k$. For $k = 2$, this is exactly the definition of LAS, so the base case is proved. If the statement above is true for $k$, then given any words 
$w^{(1)}$, $w^{(2)}$, $\ldots$, $w^{(k+1)} \in \mathcal{L}(X)$, first use LAS to define $\bar{w}^{(k)}$ with $d_H(\bar{w}^{(k)}, w^{(k)}) \leq g(|w^{(k)}|)$ so that $\bar{w}^{(k)} w^{(k+1)} \in \mathcal{L}(X)$, and then apply the inductive hypothesis to the $k$ words $w^{(1)}, w^{(2)}, \ldots, w^{(k-1)}, \bar{w}^{(k)} w^{(k+1)}$ to yield the desired $\bar{w}^{(i)}$ for $1 \leq i \leq k-1$. 
\end{proof}

%\begin{remark}
%In \cite{yamamoto}, Yamamoto also studies various weakenings of specification and their implications. The property that he calls almost specification is our almost weak specification, and the property that he calls the almost %product property is essentially our almost specification. 
%\end{remark}

Almost specification has been used in the literature to study statistical behavior such as large deviations \cite{pfister-sullivan2} and multifractal properties \cite{pfister-sullivan,thompson}.  In \cite{CT}, the first author and D.J.\ Thompson introduced a new non-uniform version of the specification property that guarantees intrinsic ergodicity, and asked whether or not almost specification could be used to prove intrinsic ergodicity.  It was shown by the second author in \cite{pavlovAS} that almost specification does not imply intrinsic ergodicity, even for $g \equiv 4$.
To prove Theorem \ref{thm:las-main}, we will show that \emph{left} almost specification with any \emph{constant} mistake function implies a version of the non-uniform specification property from \cite{climenhaga1, CT}.  This property requires the existence of $\CCC^p, \GGG, \CCC^s \subset \LLL(X)$ such that
\begin{enumerate}[label=\textup{\textbf{[\Roman{*}]}}]
\item\label{spec}
there exists $\tau\in \NN$ such that for every $v,w\in \GGG$, there exists $u\in \LLL(X)$ such that $|u|=\tau$ and $vuw\in \GGG$ (we say that ``$\GGG$ has specification''), and moreover, this $u$ has the property that $v'uw'\in \GGG$ whenever $v'\in \GGG$ is a final segment of $v$ and $w'\in \GGG$ is an initial segment of $w$.
\item\label{gap}
$h(\CCC^p \cup \CCC^s \cup (\LLL \setminus \CCC^p \GGG \CCC^s)) < h(X)$.
\end{enumerate}

In light of \ref{spec}, the collections $\CCC^p$, $\CCC^s$, and $\LLL \setminus \CCC^p \GGG \CCC^s$ are thought of as \textbf{obstructions to specification}, since words in $\CCC^p\GGG\CCC^s$ can be glued together (as in the specification property) provided we are first allowed to remove an element of $\CCC^p$ from the front of the word, and an element of $\CCC^s$ from the end. Thus
%so 
\ref{spec}--\ref{gap} can be informally stated as the requirement that ``obstructions to specification have small entropy''.
These conditions appeared in \cite{CT} (in a mildly different form), where a third condition was also required that controls how quickly words of the form $uvw$ with $u\in \CCC^p$, $v\in \GGG$, $w\in \CCC^s$ with $|u|,|w| \leq M$ can be extended to words in $\GGG$.  In our setting of LAS with bounded $g$, we produce $\CCC^p,\GGG,\CCC^s$ satisfying \ref{spec}--\ref{gap}, but it is not clear whether the collections we produce satisfy this third condition, and so we cannot apply the results from \cite{CT}.  Rather, we use the following conditions on $\GGG$ that were introduced in \cite{climenhaga1}, which we are able to verify in our setting; roughly speaking, these ask that intersections and unions of words in $\GGG$ are again in $\GGG$ (under some mild conditions).
\begin{enumerate}[label = \textup{\textbf{[III$_\mathrm{\alph{*}}${]}}}]
\item\label{inter}
There is $L\in \NN$ such that if $x\in X$ and $i\leq j\leq k\leq \ell$ are such that $k-j\geq L$ and $x([i,k]), x([j,\ell])\in \GGG$, then $x([j,k])\in \GGG$.
\item\label{union}
There is $L\in \NN$ such that if $x\in X$ and $a\leq i\leq j\leq k\leq \ell$ are such that $k-j\geq L$ and $x([i,k]), x([j,\ell]), x([a,\ell])\in \GGG$, then $x([i,\ell])\in \GGG$.
\end{enumerate}
It was shown in \cite{climenhaga1} that \ref{spec}--\ref{union} guarantee intrinsic ergodicity, as well as strong statistical properties for the unique MME, which we describe next.  We note that although \cite[Theorem 1.1]{climenhaga1} is stated using a stronger version of \ref{inter}--\ref{union} (in which no assumption on $x([a,\ell])$ is given in \ref{union}), the condition we use here is sufficient; the result we need is stated as Theorem \ref{thm:towers} below, and proved in \S\ref{sec:derive} using  \cite[Theorems 3.1(B) and 3.2]{climenhaga1}.

\subsection{Statistical properties}

%\vc{I edited this section to avoid plagiarizing \cite{climenhaga1}, and to remove some notation that we don't define}
The unique MMEs that \cite[Theorem 1.1]{climenhaga1} produces will have the following strong statistical properties.

\begin{definition}
A measure $\mu\in \MMM(X)$ satisfies an \textbf{upper Gibbs bound} if there is $Q_1>0$ such that $\mu([w]) \leq Q_1 e^{-|w| h(X)}$ for all $w\in \LLL(X)$.
\end{definition}

\begin{definition}
We say that $\mu\in \MMM(X)$ is the \textbf{limiting distribution of periodic orbits} if it is the weak* limit of the \textbf{periodic orbit measures}
\begin{equation}\label{eqn:mun}
\mu_n = \frac 1{ |\Per_n(X)|}
\sum_{x\in \Per_n(X)} \delta_x,
\end{equation}
where $\Per_n(X) = \{x\in X \mid \sigma^n x=x \}$ is the set of $n$-periodic points.
\end{definition}

\begin{definition}
The \defn{Bernoulli scheme} over a state space $S$ with probability vector $p = (p_a)_{a\in S}$ is the measure-preserving system $(S^\ZZ, \sigma, \mu_p)$, where $\sigma$ is the left shift map and $\mu_p([w]) = \prod_{i=1}^{\abs{w}} p_{a_i}$ for every $w\in S^*$.  We say that $\mu\in \MMM(X)$ has the  \defn{Bernoulli property} if $(X,\sigma,\mu)$ is measure-theoretically isomorphic to a Bernoulli scheme.
\end{definition}

\begin{definition}
A measure $\mu\in \MMM(X)$ has \defn{exponential decay of correlations} for H\"older observables if there is $\theta\in (0,1)$ such that for every pair of H\"older continuous functions $\psi_1,\psi_2 \colon X\to \RR$, there exists a constant $K(\psi_1,\psi_2)>0$ such that
\[
\abs{\int (\psi_1\circ \sigma^n) \psi_2 \,d\mu - \int \psi_1\,d\mu \int \psi_2\,d\mu} \leq K(\psi_1,\psi_2) \theta^{\abs{n}}
\text{  for every $n\in \ZZ$}.
\]
\end{definition}

\begin{definition}
Say that $\mu\in \MMM(X)$ satisfies the \defn{central limit theorem} for H\"older observables if for every H\"older continuous $\psi\colon X\to \RR$ with $\int \psi\,d\mu = 0$, the quantity $\frac 1{\sqrt n} S_n \psi = \frac 1{\sqrt n} \sum_{k=0}^{n-1} \psi\circ \sigma^k$ converges in distribution to a normal distribution $\mathcal{N}(0,\sigma_\psi)$ for some $\sigma_\psi\geq 0$; that is, if
%\begin{equation}\label{eqn:clt}
\[
\lim_{n\to\infty} \mu \left\{ x\in X \mid \frac 1{\sqrt{n}} S_n \psi(x)  \leq \tau \right\}
= \frac 1{\sigma_\psi\sqrt{2\pi}} \int_{-\infty}^\tau e^{-t^2/(2\sigma_\psi^2)} \,dt
\]
%\end{equation}
for every $\tau\in \RR$.  (When $\sigma_\psi=0$ the convergence is to the Heaviside function.)%$-1$ if $\tau<0$ and $1$ if $\tau > 0$.)

Recall that a function $\psi$ is \defn{cohomologous to a constant} if there are a measurable function $u\colon X\to \RR$ and a constant $c\in \RR$ such that $\psi(x) = u(x) - u(\sigma x) + c$ for $\mu$-a.e.\ $x\in X$.  It is typically true that in the central limit theorem the variance $\sigma_\psi^2$ is 0 if and only if $\psi$ is cohomologous to a constant.  This will hold for us as well.
\end{definition}

\subsection{A uniqueness result with non-uniform specification}

To prove Theorem \ref{thm:las-main}, we will need the following uniqueness result, which combines several theorems and remarks from \cite{climenhaga1}; see \S\ref{sec:derive} for details.

\begin{theorem}\label{thm:towers}
Let $X$ be a subshift on a finite alphabet and suppose that there are $\CCC^p, \GGG, \CCC^s \subset \LLL(X)$ satisfying \ref{spec}, \ref{gap}, \ref{inter}, and \ref{union}.  Then $X$ has a unique MME $\mu$, and $\mu$ satisfies the central limit theorem for H\"older observables.

If in addition $\tau$ from \ref{spec} is such that $\gcd \{ |w| + \tau : w\in \GGG\} = 1$, then $\mu$ is the limiting distribution of periodic orbits, the system $(X,\sigma,\mu)$ is Bernoulli, and has exponential decay of correlations for H\"older observables.
\end{theorem}

\section{Proof of Theorem \ref{thm:las-main}}\label{sec:main-pf}

To prove our main result on LAS with bounded $g$, we need to begin by showing that even if $X$ is not irreducible, we can pass to an irreducible subshift that has the same periodic points and invariant measures as $X$, and which retains the same LAS property as $X$.  We do this in \S\ref{sec:irreducible}; then in \S\ref{sec:entropy-minimal} we show that irreducibility and LAS with bounded $g$ are enough to imply entropy minimality (Theorem \ref{LASentmin}).  Finally, in \S\ref{sec:complete-the-pf} we complete the proof of Theorem \ref{thm:las-main} by producing $\CCC^p,\GGG,\CCC^s$ that satisfy conditions \ref{spec}--\ref{union} and hence allow us to apply the results on \cite{climenhaga1}. Entropy minimality is the crucial step in verifying condition \ref{gap}.

\subsection{Irreducible subshifts}\label{sec:irreducible}

Intuitively, irreducibility might ``feel'' like a weaker property than LAS with bounded $g$, since it places no restrictions on the distance required to concatenate two words.
However, LAS does not quite imply irreducibility due to a simple degenerate case; there could be words in $\mathcal{L}(X)$ which can only appear finitely many times in points of $X$. For instance, take $X = \{x \in \{0,1\}^{\mathbb{Z}}  :  |\{n  :  x(n) = 1\}| \leq 1\}$, the set of $0$-$1$ sequences with at most one $1$. This is clearly not irreducible, but satisfies LAS with $g \equiv 1$; any word can be changed to a string of $0$s with at most one change, and a string of $0$s can legally precede any word in $\mathcal{L}(X)$. 

To avoid such situations, we pass to a natural subshift: recall that the \defn{measure center} $\hat{X} \subset X$ is the union of the supports of all invariant measures on $X$.  This a closed shift-invariant subset of $X$, which has the same simplex of invariant measures as $X$ by definition, and hence also has the same set of periodic points (since each periodic orbit supports a unique invariant measure.)

\begin{theorem}\label{measure-center}
%Suppose $X$ has left almost specification with mistake function $g(n)$, and let $\hat{X}$ be the measure center of $X$.  Then $\hat{X}$ has left almost specification with mistake function $g(n)$ and is irreducible.
If $X$ has LAS with mistake function $g(n)$, then its measure center $\hat{X}$ is irreducible and has LAS with mistake function $g(n)$.
\end{theorem}
\begin{proof}
We start by using LAS to get a preliminary version of irreducibility.

\begin{lemma}\label{lem:u-to-v}
For every $u\in \LLL(\hat{X})$ there is $N=N_u\in \NN$ such that for every $v\in \LLL(X)$ there is a word $z\in A^*$ such that $|z| \leq N$ and $uzv\in \LLL(X)$.
\end{lemma}
\begin{proof}
Since $u\in \LLL(\hat{X})$ there is an ergodic measure $\mu$ such that $\mu([u])>0$.  Let $\epsilon= \frac 12\mu([u])$.  By the ergodic theorem, there is $N_0\in \NN$ such that for all $N\geq N_0$ there is $w\in \LLL_N(X)$ with
\begin{equation}\label{eqn:many-us}
\#\{1\leq k\leq |w|-|u| : w([k,k+|u|)) = u \} \geq \epsilon N.
\end{equation}
Since $\frac{g(n)}{n} \rightarrow 0$, we may choose $N\geq N_0$ such that $g(N) < \epsilon N / |u|$. Fix $w\in \LLL_N(X)$ satisfying \eqref{eqn:many-us}.  Then every $\bar{w}\in B_{g(|w|)}(w)$ has $u$ as a subword, since changing $g(|w|)$ symbols destroys at most $|u| g(|w|) = |u| g(N) < \epsilon N$ occurrences of $u$; thus every such $\bar{w}$ has the form $yuz$ for some $y,z\in A^*$.  By LAS for $X$ there is some such $\bar{w}$ with $\bar{w} v = yuzv \in \LLL(X)$, and thus $uzv \in \LLL(X)$, which proves the lemma.
\end{proof}

Now we use Lemma \ref{lem:u-to-v} iteratively to prove Theorem \ref{measure-center}.  Given $u,v\in \LLL(\hat{X})$, let $N = \max\{N_u,N_v\}$.  Use LAS to get $\bar{u}v\in \LLL(X)$ for some $\bar{u} \in B_{g(|u|)}(u)$, and apply Lemma \ref{lem:u-to-v} to the words $u$ and $\bar{u}v$ to get $z^{(1)}\in A^*$ such that $|z^{(1)}| \leq N$ and $uz^{(1)}\bar{u}v\in \LLL(X)$.  Then apply Lemma \ref{lem:u-to-v} to the words $v$ and $uz^{(1)}\bar{u}v$ to obtain $y^{(1)} \in A^*$ such that $|y^{(1)}| \leq N$ and $vy^{(1)}uz^{(1)}\bar{u}v$.  Iterating, we obtain $y^{(i)},z^{(i)} \in A^*$ with lengths $\leq N$ such that
\[
w^{(j)} := vy^{(j)}uz^{(j)}\bar{u}vy^{(j-1)}uz^{(j-1)}\bar{u}v\cdots vy^{(1)}uz^{(1)}\bar{u}v \in \LLL(X) \text{ for all } j\in \NN.
\]
By compactness there is $x\in X$ such that
\[
x_{(-\infty,0]} = \cdots vy^{(j)}uz^{(j)} \bar{u}v \cdots vy^{(1)}uz^{(1)}\bar{u}v.
\]
Let $\mu_n = \frac 1n \sum_{k=0}^{n-1} \delta_x \circ \sigma^k$.  Each $z^{(j)}$ is chosen from the finite set $\bigcup_{n=0}^N A^n$, so there is $z\in \bigcup_{n=0}^N A^n$ for which $\{j\in \NN : z^{(j)} = z\}$ has positive upper density, and thus there is some subsequence $n_i\to \infty$ such that $\liminf_i \mu_{n_i}([uz\bar{u}v])>0$.  By passing to a further subsequence if necessary, we can consider the weak* limit $\mu$ of the sequence $\mu_{n_i}$, which is $\sigma$-invariant by standard arguments and has $\mu([uz\bar{u}v])>0$ by our choice of $n_i$. We conclude that $uz\bar{u}v\in \LLL(\hat{X})$, implying irreducibility. Clearly the subword $\bar{u}v \in \LLL(\hat{X})$ as well, which implies LAS and completes the proof of Theorem \ref{measure-center}.
\end{proof}

\begin{remark}
The proof of Lemma \ref{lem:u-to-v} shows that for every $u\in \LLL(\hat{X})$, there is a word $w\in \LLL(X)$ that contains at least $g(|w|)+1$ disjoint instances of $u$ as subwords, and the paragraphs below that lemma show that this property in fact characterizes words in $\LLL(\hat{X})$.  In particular, when $g\equiv m$,  we can write $\mathcal{O}$ for the set of words $u\in \LLL(X)$ such that no point of $X$ contains $m+1$ disjoint occurrences of $u$, and conclude that $\hat{X}$ is the subshift of points of $X$ containing no word from $\mathcal{O}$.%\footnote{In a previous version of this paper (\texttt{arXiv:1605.05354v1}) we gave an alternate proof that the shift defined by this condition fulfills the conclusion of Theorem \ref{measure-center}; we are grateful to the anonymous referee for pointing out the description as the measure center and suggesting the current more conceptual proof.}
\end{remark}

\subsection{Entropy minimality}\label{sec:entropy-minimal}

%\vc{Here and in many places throughout the rest of the paper, we had written $L_n(X)$ instead of $\LLL_n(X)$.  I've changed it to $\mathcal{L}_n(X)$ using the macro \texttt{$\backslash$LL} so that we can easily reverse it if it was supposed to be different.}
Before proving that irreducibility and LAS with bounded $g$ are enough to obtain entropy minimality, we need to establish some counting estimates.  Recall from Lemma \ref{wordcount} that $|\LL_n(X)| \geq e^{nh(X)}$ for every $n$.  In general the definition of $h(X)$ gives the upper bound $|\LL_n(X)| \leq C_n e^{nh(X)}$ for some subexponential $C_n$.  An important part of the uniqueness proof in \cite{rB74,CT} is to prove that $C_n$ can be taken to be bounded.  At this stage of our proof we do not yet get quite this bound, but we can prove that $C_n$ grows at most polynomially for subshifts with almost specification; note that by Lemma~\ref{LAS=>AS}, the same is true for subshifts with LAS.

%Almost specification (and therefore LAS as well) implies the following upper bound on the word complexity function, which we will soon use to prove entropy minimality.

\begin{lemma}\label{upperbound}
If $X$ is a subshift with almost specification with constant $g \equiv m$, then for every $n$,
\begin{equation}\label{upperboundeq}
|\LL_n(X)| \leq |A|^{2m} n^{2m} e^{n h(X)}.
\end{equation}
\end{lemma}
\begin{proof}
%{This was originally written as a proof by contradiction, which seems more convoluted than necessary.  (Also as a matter of principle I prefer to avoid proofs by contradiction whenever possible...)}
Suppose that $X$ is a subshift satisfying AS with $g \equiv m$.
%the LAS case is then immediate since LAS with $g = m$ implies AS with $g = m$), and for a contradiction assume that there exists $n$ with $|\LL_n(X)| > |A|^{2m} n^{2m} e^{n h(X)}$. 
Choose a maximal $(2m+1)$-separated subset $S_n \subset \LL_n(X)$ with respect to Hamming distance; since the cardinality of a Hamming ball of radius $2m$ is clearly bounded from above by $|A|^{2m} n^{2m}$, we have
\begin{equation}\label{eqn:Sn}
|S_n| \geq |\LL_n(X)|\cdot |A|^{-2m} n^{-2m}. %> e^{n h(X)}$. 
\end{equation}
For any $k\in \NN$, almost specification lets us define $f\colon (S_n)^k \to \LLL_{kn}(X)$ by
%Then, for any $k$ and any collection of words $w_i \in S$, $1 \leq i \leq k$, use almost specification to construct 
$f(w^{(1)}, \ldots, w^{(k)}) = \bar{w}^{(1)} \ldots \bar{w}^{(k-1)} \bar{w}^{(k)} \in \LL_{kn}(X)$ where
$\bar{w}^{(i)} \in B_m(w^{(i)})$ for $1 \leq i \leq k$. Then the map $f$ is clearly injective, since for any $s \neq s'$ in $S$, the Hamming balls $B_m(s)$ and $B_m(s')$ are disjoint. 
%However, this implies that for every $k$,
Using \eqref{eqn:Sn}, this gives
\begin{align*}
|\LL_{kn}(X)| &\geq |S_n|^k \geq |\LL_n(X)|^k|A|^{-2mk}n^{-2mk}, \textrm{ and so} \\
\tfrac 1{kn} \ln|\LL_{kn}(X)| &\geq \tfrac 1n\ln|\LL_n(X)| - \tfrac {2m}n \ln(|A|n).
\end{align*}
%Taking logarithms of both sides, dividing by $kn$, and letting $k \rightarrow \infty$ yields
Sending $k\to\infty$ gives %$h(X) \geq \frac 1n \ln|\LL_n(X)| - \frac{2m}n\ln(|A|n)$
$\frac 1n\ln|\LL_n(X)| \leq h(X) + \frac {2m}n\ln(|A| n)$; then multiplying by $n$ and taking exponentials gives \eqref{upperboundeq}.
\end{proof}

%\[h(X) \geq \frac{1}{n} \log |S| > h(X),\]
%an obvious contradiction. Therefore, our original assumption was false and the lemma is proved.

The following technical theorem will be a main tool in our proof that LAS and irreducibility imply entropy minimality, and may be of some independent interest.
It shows that shifts with entropy minimality satisfy a sort of weakened Gibbs counting bound on the number of words that end with a given word $w$, as long as the overall word complexity function satisfies a polynomial upper bound of the sort just proved.

\begin{theorem}\label{entminthm}
Suppose that $X$ is a subshift for which there exist $C,r>0$ so that $|\LL_n(X)| \leq Cn^r e^{n h(X)}$ for every $n$, and that $\mu$ is a measure of maximal entropy on $X$. Then, for every $w \in \mathcal{L}(X)$ with $\mu([w]) > 0$, there exists $\epsilon > 0$ so that $|\LL_n(X) \cap \mathcal{L}(X)w| > \epsilon n^{-1/2} e^{n h(X)}$ for all $n \geq |w|$.
\end{theorem}

\begin{proof}
%\vc{I introduced some new notation that I think makes this proof easier to read (especially the long computations from before appear simpler now).  I also clarified a couple steps that didn't seem obvious to me.}
Suppose that $X$, $C$, $r$, $\mu$, and $w$ are as in the theorem, and denote by $A$ the alphabet of $X$. By ergodic decomposition, we may assume without loss of generality that $\mu$ is ergodic. Use the ergodic theorem to choose $N$ so that $\mu\big(\bigcup_{i = 0}^{N-1} \sigma^i [w] \big) > \frac{2r}{2r+1}$.

Then, for each $n \geq N + |w|$, define $S_n = \mathcal{L}_n(X) \cap \bigcup_{i = n - |w| - N + 1}^{n - |w|} \sigma^{-i} [w]$, the set of $n$-letter words in $\mathcal{L}(X)$ which contain a $w$ somewhere in the last $|w| + N - 1$ letters. We use the notation $\mu(S_n)$ to refer to $\mu\left(\bigcup_{v \in S_n} [v]\right)$; then $\mu(S_n) > \frac{2r}{2r+1}$ for all $n > N + |w|$. Note that this implies $\frac{1 - \mu(S_n)}{\mu(S_n)} < \frac{1}{2r}$.

Now, we will apply Theorem 4.7 from \cite{pavlovgap}. That result is phrased in terms of general expansive systems (not just subshifts) and for topological pressure for a general potential $\phi$ (not just topological entropy.) However, once restricting to our setting (by setting $\phi = 0$), it says that there exists a constant $M$ so that
\[
|S_n| \geq M (e^{nh(X)})^{1/\mu(S_n)} |\LL_n(X)|^{1 - (1/\mu(S_n))} 2^{-1/\mu(S_n)}.
\]
Since $|\LL_n(X)| \leq Cn^r e^{nh(X)}$ and $\mu$ is an MME on $X$, this implies that
\begin{multline}\label{wordbound}
|S_n| \geq M e^{nh(X)} (Cn^r)^{-(1 - \mu(S_n))/\mu(S_n)} 2^{-1/\mu(S_n)} \\ > Me^{nh(X)} (Cn^r)^{-1/(2r)} 2^{-(2r+1)/2r} = \epsilon' n^{-1/2} e^{nh(X)}
\end{multline}
for $\epsilon' = M C^{-1/(2r)} 2^{-(2r+1)/2r}$.

We now note that since each word in $S_n$ ends with $w$, followed by a word of length less than $N$, we have
%the quantity $|S_n|$ is bounded from above by
\begin{equation}\label{wordbound2}
|S_n| \leq 
\sum_{i=0}^{N-1} |\LL_{n-i}(X) \cap \mathcal{L}(X)w| |A|^i \leq N |A|^{N-1} |\LL_n(X) \cap \mathcal{L}(X)w|.
\end{equation}
(For the inequality, note that every word in $\LL_{n-i}(X) \cap \mathcal{L}(X)w$ can be extended on the left to a word in $\LL_{n}(X) \cap \mathcal{L}(X)w$.)  Now, combining (\ref{wordbound}) and (\ref{wordbound2}) yields
\begin{equation}\label{wordbound3}
|\LL_n(X) \cap \mathcal{L}(X)w| > \frac{\epsilon'}{N |A|^{N-1}} n^{-1/2} e^{nh(X)}.
\end{equation}
The desired bound has now been shown for $n \geq N + |w|$ (and $\epsilon = \frac{\epsilon'}{N |A|^{N-1}}$), and since $|\LL_{n}(X) \cap \mathcal{L}(X)w| > 0$ for the finite set of $n$ in $[|w|, N + |w|)$, such a bound trivially still holds for $n \geq |w|$.
\end{proof}

\begin{remark}
In fact Theorem~\ref{entminthm} could be generalized by using any upper bound of the form $|\LL_n(X)| \leq f(n) e^{n h(X)}$ as a hypothesis and deriving, for any $k$, a lower bound of the form $|\LL_s(X) \cap \mathcal{L}(X)w| > \epsilon f(n)^{-1/k} e^{n h(X)}$ as a conclusion (in the choice of $N$ at the start of the proof, replace $2r$ with $k$). In particular, we note that we could bound $|\LL_n(X) \cap \mathcal{L}(X)w|$ from below by some constant multiple of $n^{-\beta} e^{nh(X)}$ for any desired $\beta > 0$; for our purposes, $\beta = 1/2$ will suffice.
%The exponent $-1/2$ in Theorem \ref{entminthm} can be replaced by any negative number $-\beta$ by replacing $n^{(q-1)/2}$ and $n^{q/2}$ with $n^{(q-1)\beta}$ and $n^{q\beta}$ in the first paragraph of the proof.  For our purposes, $\beta = 1/2$ suffices.
\end{remark}

The following corollary is immediate, since every MME for an entropy minimal subshift is fully supported.

\begin{corollary}\label{entmincor}
Suppose that $X$ is an entropy minimal subshift and that there exist $C,r>0$ so that $|\LL_n(X)| \leq Cn^r e^{n h(X)}$ for every $n$. Then, for every $w \in \mathcal{L}(X)$, there exists $\epsilon > 0$ so that $|\LL_n(X) \cap \mathcal{L}(X)w| > \epsilon n^{-1/2} e^{n h(X)}$ for all $n \geq |w|$.
\end{corollary}

Now we can prove that LAS (with bounded $g$) and irreducibility together imply entropy minimality.

\begin{theorem}\label{LASentmin}
If $X$ is an irreducible subshift with left almost specification with constant $g \equiv m$, then $X$ is entropy minimal.
\end{theorem}
\begin{remark}\label{rmk:unbdd}
This result fails if $g$ is unbounded; see \S\ref{sec:unbdd}.  This is the point at which our proof breaks down for unbounded mistake functions.
\end{remark}
\begin{proof}[Proof of Theorem \ref{LASentmin}]
Take $X$ a subshift as in the theorem. By Proposition 2.8 of \cite{QT}, $X$ contains an entropy minimal subshift $Y$ with $h(Y) = h(X)$. We suppose for a contradiction that $Y$ is a proper subset of $X$, and so there exists 
$u \in \mathcal{L}(X) \setminus \mathcal{L}(Y)$. Define $i$ to be the maximum number of changes necessary to a word in $\mathcal{L}(Y)$ when using left almost specification to append a word of $\mathcal{L}(X)$ to its right and yield a word in $\mathcal{L}(X)$, i.e.
\[
i = \max_{v \in \mathcal{L}(Y), w \in \mathcal{L}(X)} \min \{j \ : \ \exists v' \in B_j(v) \textrm{ for which } v'w \in \mathcal{L}(X)\}.
\]
Clearly $i \leq m$ since $X$ has LAS with $g \equiv m$. Choose words $v \in \mathcal{L}(Y)$ and $w \in \mathcal{L}(X)$ achieving the maximum $i$, i.e.\ for $j < i$, there exists no 
$y \in B_{j}(v)$ for which $yw \in \mathcal{L}(X)$. 

By Lemma~\ref{upperbound}, there exist $C,r>0$ such that for all $n$, we have
$|\LL_n(X)| \leq C n^r e^{n h(X)}$. Since $Y \subseteq X$, the same bound holds for $|\LL_n(Y)|$, and since $Y$ is entropy minimal, we can apply Corollary~\ref{entmincor} to $Y$. Since $h(X) = h(Y)$, this yields $\epsilon > 0$ so that for all $n \geq |v|$,
\begin{equation}\label{eqn:Ms}
M_n := |\LL_n(Y) \cap \mathcal{L}(Y) v| > \epsilon n^{-1/2} e^{n h(Y)} = \epsilon n^{-1/2} e^{n h(X)},
\end{equation}
where $v$ is the `maximally changing word' from above. 

By irreducibility of $X$, define 
$z \in \mathcal{L}(X)$ containing $m+1$ disjoint occurrences of $w$, followed by $m+1$ disjoint occurrences of $u$. 
Denote by $N$ the length of $z$, and assume without loss of generality that $N \geq |v|$.
Then by (\ref{eqn:Ms}), for all $i \in \mathbb{N}$, $|M_{iN}| > \epsilon (iN)^{-1/2} e^{iN h(X)}$.
Since $t!$ grows superexponentially, we may fix $t$ such that $\sqrt{t!} > 2(\sqrt{N}e^{Nh(X)}/\epsilon)^t$. Write $M = \sum_{i = 1}^t iN$.

We will make many words in $\LL(X)$ by using left almost specification to almost concatenate words in various $\LL_{iN}(Y)$ with copies of 
$z$ in an alternating fashion. Specifically, for any $k$, we create words in $\LL_{k(M+tN)}(X)$ in the following way: 
\begin{enumerate}[label=(\alph{*})]
\item choose any permutations
$\pi_1, \pi_2, \ldots, \pi_k$ of $\{1,\ldots,t\}$ and words $w^{(i,j)} \in \LL_{iN}(Y) \cap \mathcal{L}(Y) v$ for $1 \leq i \leq t$ and $1 \leq j \leq k$;
\item\label{ul}
define words $u^{(\ell)}$ for $1\leq \ell \leq kt$ by $u^{(i + t(j-1))} = w^{(\pi_j(i),j)}$ for $1\leq i \leq t$ and $1\leq j \leq k$, so 
%$u_i$, $1 \leq i \leq kt$, as follows:
%$u_1 = w_{\pi_1(1),1}$, $u_2 = w_{\pi_1(2),1}$, \ldots, $u_t = w_{\pi_1(t),1}$, $\ldots$, $u_{kt - 1} = w_{\pi_k(t-1),k}$, 
%$u_{kt} = w_{\pi_k(t),k}$. In other words, 
the list $u^{(1)}$, $u^{(2)}$, $\ldots$, $u^{(kt)}$ consists of the words $w^{(i,1)}$ permuted according to $\pi_1$, 
followed by the words $w^{(i,2)}$ permuted according to $\pi_2$, and so on;
\item given any $\ppi = (\pi_1,\dots, \pi_k)$ and $\ww = (w^{(i,j)})$, define a word $f(\ppi,\ww)\in \LL_{k(M+tN)}(X)$ using LAS iteratively on $u^{(1)},\dots, u^{(\ell)}$ from right to left, as described in detail below.
\end{enumerate}
After defining $f(\ppi,\ww)$, we will show that the map $f$ is injective, which leads to a lower bound on $\LL_{k(M+tN)}(X)$ and hence a lower bound on $h(X)$.  By an appropriate choice of the parameter $t$, we will be able to show that $h(X)>h(Y)$ and hence obtain a contradiction.

Now we define $f(\ppi,\ww) \in \LL_{k(M+tN)}(X)$.
%Define $f(\pi_1, \ldots, \pi_k, w_{1,1}, \ldots, w_{t,k}) \in \LL_{k(M+tN)}(X)$ as follows. 
Begin with $u^{(tk)}$, then use left almost specification to find $z^{(tk-1)} \in B_m(z)$ for which $z^{(tk-1)} u^{(tk)} \in \mathcal{L}(X)$. Since $z$ contained $m + 1$ occurrences of $w$, at least one of them remains; delete the portion to the left of the leftmost
remaining occurrence out of those to create a word $\bar{z}^{(tk-1)}$ with length less than or equal to $N$. Then%\vc{Moved these to an itemized list to make it easier to read, it was a very long unbroken paragraph before.}
\begin{itemize}
\item $\bar{z}^{(tk-1)} u^{(tk)} \in \mathcal{L}(X)$, 
\item $\bar{z}^{(tk-1)}$ begins with $w$, and 
\item $\bar{z}^{(tk-1)}$ contains $u$ (since $z$ had $m+1$ disjoint instances of $u$, and
the part deleted to get $\bar{z}^{(tk-1)}$ from $z^{(tk-1)}$ was to the left of all of them). 
\end{itemize}
Extend $u^{(tk-1)}$ on the left by $N - |\bar{z}^{(tk-1)}|$ letters to make a new word $\bar{u}^{(tk-1)} \in \mathcal{L}(Y)$, which still ends with $v$. Use left almost specification to find $U^{(tk-1)} \in B_i(\bar{u}^{(tk-1)})$ for which $U^{(tk-1)} \bar{z}^{(tk-1)} u^{(tk)} \in \mathcal{L}(X)$; we needed less than or equal to $i$ changes to $\bar{u}^{(tk-1)}$ since $\bar{u}^{(tk-1)} \in \mathcal{L}(Y)$. However, recall that $\bar{u}^{(tk-1)}$ ends with $v$ and $\bar{z}^{(tk-1)}$ begins with $w$, and that at least $i$ changes are required to $v$ to concatenate $w$ on its right. Therefore, all $i$ of the changes made in changing $\bar{u}^{(tk-1)}$ to $U^{(tk-1)}$ occurred within the terminal copy of $v$.
 
Observe that the length of $U^{(tk-1)} \bar{z}^{(tk-1)} u^{(tk)}$ is equal to $|u^{(tk-1)}| + N + |u^{(tk)}|$; in other words, even though 
$\bar{z}^{(tk-1)}$ may have length smaller than $N$, $U^{(tk-1)}$ ``corrects'' the loss so that the length of $U^{(tk-1)} \bar{z}^{(tk-1)} u^{(tk)}$ is forced.
Continuing in this fashion we eventually arrive at a word 
\begin{equation}\label{eqn:fpiw}
%f(\pi_1, \ldots, \pi_k, w_{1,1}, \ldots, w_{t,k}) = 
f(\ppi,\ww)=
U^{(1)} \bar{z}^{(1)} U^{(2)} \bar{z}^{(2)} \ldots U^{(tk-1)} \bar{z}^{(tk-1)} u^{(tk)} \in \LL_{k(M+tN)}(X).
\end{equation}
At each step, let's say we choose the lexicographically minimal possible word to append, so that the procedure described is deterministic.

We claim that $f$ is injective. Suppose
$(\ppi_1,\ww_1) \neq (\ppi_2,\ww_2)$,
% $(\pi^{(1)}_1, \ldots, \pi^{(1)}_k, w^{(1)}_{1,1}, \ldots, w^{(1)}_{t,k}) \neq 
%(\pi^{(2)}_1, \ldots, \pi^{(2)}_k, w^{(2)}_{1,1}, \ldots, w^{(2)}_{t,k})$. 
that $(\ppi_1,\ww_1)$
%$(\pi^{(1)}_1, \ldots, \pi^{(1)}_k, w^{(1)}_{1,1}, \ldots, w^{(1)}_{t,k})$ 
induces $u_1^{(1)}, \ldots, u_1^{(tk)}$ as in \ref{ul} above,
%the description of $f$, 
and that $u_2^{(1)}, \ldots, u_2^{(tk)}$ is similarly induced by $(\ppi_2,\ww_2)$.
%$(\pi^{(2)}_1, \ldots, \pi^{(2)}_k, w^{(2)}_{1,1}, \ldots, w^{(2)}_{t,k})$. 
We break into two cases.\\

\noindent
\emph{Case 1:} $\ppi_1 = \ppi_2$.
%\noindent
%\textbf{Case 1:} $(\pi^{(1)}_1, \ldots, \pi^{(1)}_k) = (\pi^{(2)}_1, \ldots, \pi^{(2)}_k)$. 

In this case, $|u_1^{(\ell)}| = |u_2^{(\ell)}|$ for $1 \leq \ell \leq tk$. We may then choose maximal $\ell$ so that 
$u_1^{(\ell)} \neq u_2^{(\ell)}$ (such an $\ell$ must exist since
$(\ppi_1,\ww_1) \neq (\ppi_2,\ww_2)$).
 %$(\pi^{(1)}_1, \ldots, \pi^{(1)}_k, w^{(1)}_{1,1}, \ldots, w^{(1)}_{t,k}) \neq (\pi^{(2)}_1, \ldots, \pi^{(2)}_k, w^{(2)}_{1,1}, \ldots, w^{(2)}_{t,k})$). 

During the creation of $f(\ppi_1,\ww_1)$ and $f(\ppi_2,\ww_2)$,
%$f(\pi^{(1)}_1, \ldots, \pi^{(1)}_k, w^{(1)}_{1,1}, \ldots, w^{(1)}_{t,k})$ and \newline $f(\pi^{(2)}_1, \ldots, \pi^{(2)}_k, w^{(2)}_{1,1}, \ldots, w^{(2)}_{t,k})$, 
the first $tk - \ell - 1$ steps will be identical and will yield the same word $\bar{z}^{(\ell + 1)} U^{(\ell + 1)} \ldots \bar{z}^{(tk-1)} u^{(tk)}$. Then, at the next step, the unequal words $u_1^{(\ell)}$ and $u_2^{(\ell)}$ are extended to the left by $N - |\bar{z}^{(\ell + 1)}|$ units, which yields unequal words $\bar{u}^{(\ell)}_1$ and $\bar{u}^{(\ell)}_2$, and then left almost specification is used to find 
$U_1^{(\ell)}$ and $U_2^{(\ell)}$ which can be appended to the left of $\bar{z}^{(\ell + 1)} U^{(\ell + 1)} \ldots \bar{z}^{(tk-1)} u^{(tk)}$. However, as noted above, the only changes made are inside the terminal copies of $v$ within $\bar{u}^{(\ell)}_1$ and $\bar{u}^{(\ell)}_2$, and so the appended words 
$U_1^{(\ell)}$ and $U_2^{(\ell)}$ are still unequal, meaning that the final words
%$f(\pi^{(1)}_1, \ldots, \pi^{(1)}_k, w^{(1)}_{1,1}, \ldots, w^{(1)}_{t,k})$ and $f(\pi^{(2)}_1, \ldots, \pi^{(2)}_k, w^{(2)}_{1,1}, \ldots, w^{(2)}_{t,k})$ 
$f(\ppi_1,\ww_1)$ and $f(\ppi_2,\ww_2)$
are also unequal.\\

\noindent
\emph{Case 2:} $\ppi^{(1)} \neq \ppi^{(2)}$.
%\noindent
%\textbf{Case 2:} $(\pi^{(1)}_1, \ldots, \pi^{(1)}_k) \neq (\pi^{(2)}_1, \ldots, \pi^{(2)}_k)$.

In this case, there must exist $\ell$ so that $|u_1^{(\ell)}| \neq |u_2^{(\ell)}|$; take $\ell$ to be maximal with this property.
Without loss of generality assume $|u_1^{(\ell)}| > |u_2^{(\ell)}|$. By assumption on $S$, 
$|u_2^{(\ell)} - u_1^{(\ell)}| \geq N$.
Again, in the process of creating 
%$f(\pi^{(1)}_1, \ldots, \pi^{(1)}_k, w^{(1)}_{1,1}, \ldots, w^{(1)}_{t,k})$ and $f(\pi^{(2)}_1, \ldots, \pi^{(2)}_k, w^{(2)}_{1,1}, \ldots, w^{(2)}_{t,k})$, 
$f(\ppi_1,\ww_1)$ and $f(\ppi_2,\ww_2)$,
the first $tk - \ell - 1$ steps will be identical, yielding the same word 
$\bar{z}^{(\ell + 1)} U^{(\ell + 1)} \ldots \bar{z}^{(tk-1)} u^{(tk)}$. 

Then, in the construction of $f(\ppi_2,\ww_2)$,
%$f(\pi^{(2)}_1, \ldots, \pi^{(2)}_k, w^{(2)}_{1,1}, \ldots, w^{(2)}_{t,k})$, 
once 
$U_2^{(\ell)}$ is appended to the left of $\bar{z}^{(\ell + 1)} U^{(\ell + 1)} \ldots \bar{z}^{(tk-1)} u^{(tk)}$,
somewhere in the $N$ letters immediately to the left an occurrence of $u$ will be added at the next step (inside some 
$\bar{z}^{(\ell)}_2 \in B_m(z)$). However, the corresponding locations in $U_1^{(\ell)} \bar{z}^{(\ell + 1)} U^{(\ell + 1)} \ldots \bar{z}^{(tk-1)} u^{(tk)}$ are part of $U_1^{(\ell)} \in \mathcal{L}(Y)$ since $|u_1^{(\ell)}| \geq |u_2^{(\ell)}| + N$, and so contain no occurrences of $u$ since $u \notin \mathcal{L}(Y)$. Thus there is a location at which $f(\ppi_2,\ww_2)$
%$f(\pi^{(2)}_1, \ldots, \pi^{(2)}_k, w^{(2)}_{1,1}, \ldots, w^{(2)}_{t,k})$
contains a $u$ and $f(\ppi_1,\ww_1)$
%$f(\pi^{(1)}_1, \ldots, \pi^{(1)}_k, w^{(1)}_{1,1}, \ldots, w^{(1)}_{t,k})$ 
does not, meaning that they are unequal.\\

We have shown that $f$ is one-to-one, and so it generates a set within $\LL_{k(M + tN)}(X)$ with cardinality at least
$(t! \prod_{i=1}^t M_{iN})^k$, where we recall that $M_{iN}$ was defined in \eqref{eqn:Ms} and satisfies $M_{iN} \geq \epsilon (iN)^{-1/2} e^{iN h(X)}$.  Thus
\begin{multline}\label{entminbound3}
|\LL_{k(M+tN)}(X)| \geq \left(t! \prod_{i=1}^t M_{iN}\right)^k \geq \left(t! \prod_{i=1}^t \epsilon (iN)^{-1/2} e^{iN h(X)}\right)^k\\
= (e^{Mh(X)} \sqrt{t!} (\epsilon/\sqrt{N})^t)^k > 2^k e^{k(M+tN)h(X)}.
\end{multline}
(Recall that $t$ was chosen so that $\sqrt{t!} > 2(\sqrt{N}e^{Nh(X)}/\epsilon)^t$.)

Now, we take logarithms in (\ref{entminbound3}), divide by $k(M + tN)$, and let $k \rightarrow \infty$, which implies that $h(X) \geq h(Y) + \frac{\log 2}{M + tN}$, contradicting the earlier statement that $h(Y)=h(X)$. Therefore, our original assumption that $Y\neq X$ was false, and we conclude that $Y = X$; since 
$Y$ was entropy minimal, we have shown that $X$ is entropy minimal.
\end{proof}

\subsection{Completion of the proof}\label{sec:complete-the-pf}

%\textcolor{red}{We can now finally prove our main result.}
Now we use the results from \S\ref{sec:irreducible} and \S\ref{sec:entropy-minimal} to complete the proof of Theorem \ref{thm:las-main}.  By Theorem \ref{measure-center}, we may assume without loss of generality that $X$ is irreducible, since all desired properties are dependent only on periodic points of $X$ and measures on $X$. Define any $m$ so that $g(n) \leq m$ for all $n$.

Define $i$ to be the maximum number of changes necessary to a word in $\mathcal{L}(X)$ when using left almost specification to append a word of $\mathcal{L}(X)$ to its right and yield a word in $\mathcal{L}(X)$, i.e.
\begin{equation}\label{eqn:i}
i = \max_{y \in \mathcal{L}(X), v \in \mathcal{L}(X)} \min \{j \ : \ \exists y' \in B_j(y) \textrm{ for which } y'v \in \mathcal{L}(X)\}.
\end{equation}
We have $i \leq m$ from the LAS property. Choose words $y \in \mathcal{L}(X)$ and $v^{(0)} \in \mathcal{L}(X)$ achieving this maximum, i.e.\ for $j < i$, there exists no $y' \in B_{j}(y)$ for which $y'v^{(0)} \in \mathcal{L}(X)$.  
Let $\LLL^y = \{w\in \LLL : wy\in \LLL \}$.
Given $w\in \LLL^y$ and $v\in \LLL$, let $D(w,v) = \{y' \in B_i(y) : wy'v \in \LLL\}$; this is non-empty because LAS gives $x\in B_i(wy)$ such that $xv\in \LLL$, and by our choice of $y$ and $v$, all the changes between $wy$ and $x$ must happen in the last $|y|$ symbols, so $x=wy'$ for some $y'\in B_i(y)$.

Observe that if $v'\in v\LLL\cap \LLL$ and $w'\in \LLL w\cap \LLL^y$, then $D(w',v') \subset D(w,v)$.  Now define a sequence of words $w^{(k)},v^{(k)}$ for $k\geq 1$ iteratively as follows: If there are $w'\in \LLL w^{(k-1)} \cap \LLL^y$ and $v' \in v^{(k-1)}\LLL \cap \LLL$ such that $D(w',v') \neq D(w^{(k-1)},v^{(k-1)})$, then let $w^{(k)} = w'$ and $v^{(k)} = v'$; if there are no such $v',w'$, terminate the construction.  Because the sets $D(w^{(k)},v^{(k)})$ are finite, nested, and non-empty, the construction terminates at some $k$; putting $v=v^{(k)}$ and choosing $y'\in D(w^{(k)},v)$, we see that for every $x\in \LLL^y \cap \LLL w^{(k)}$ and $z\in \LLL \cap v\LLL$, we have $xy'z\in \LLL$.  We will write $u=w^{(k)}y$, $u' = w^{(k)}y'$, and use this in the following form:
\begin{equation}\label{eqn:uu'}
\text{if } xu\in \LLL \text{ and } z\in v\LLL\cap \LLL \text{ then } xu'z\in \LLL.
\end{equation}

Consider the collection
\[
\mathcal{G} = \{w \in \mathcal{L}(X) \cap v \mathcal{L}(X) \ : \ wu \in \mathcal{L}(X)\}
\]
of words in $\mathcal{L}(X)$ starting with $v$ which can be legally followed by $u$. %The following lemma is not needed for uniqueness, but will let us rule out any `periodicity' in the Bernoulli property and decay of correlations. 

To apply the results from \cite{climenhaga1} we define the ``prefix'' and ``suffix'' collections 
\begin{align*}
\mathcal{C}^{p} &= \{w \in \mathcal{L}(X) \ : \ w \textrm{ does not contain $v$ as a subword}\} \\
\mathcal{C}^{s} &= \{w \in \mathcal{L}(X) \cap u \mathcal{L}(X) \ : \ w \textrm{ contains $u$ only once as a subword}\}.
\end{align*}
Let $\CCC^p \GGG \CCC^s$ denote the set of words that can 
be decomposed as a concatenation of a word in $\mathcal{C}^{p}$, followed by a word in $\mathcal{G}$, followed by a word in $\mathcal{C}^{s}$, and let $\BBB := \LLL \setminus \CCC^p \GGG \CCC^s$ be the set of words that do not admit such a decomposition.
%Take $\mathcal{B} = \LLL \setminus \CCC^p \GGG \CCC^s$ to denote the set of words which cannot be decomposed as a concatenation of a word in $\mathcal{C}^{p}$, followed by a word in $\mathcal{G}$, followed by a word in $\mathcal{C}^{s}$. 
Finally, define $\mathcal{C} = \mathcal{C}^{p} \cup \mathcal{C}^{s} \cup \mathcal{B}$. 

Condition \ref{spec} says that $\GGG$ has specification, and follows from our choice of $\GGG$ and
\eqref{eqn:uu'}.  Given any $w,w'\in \GGG$, we have $wu,w'u \in \LLL\cap v\LLL$ by the definition of $\GGG$, and so \eqref{eqn:uu'} gives $wu'w'u\in \LLL$.  Moreover, $wu'w'$ begins with $v$ (since $w\in \GGG$) and so $wu'w'\in \GGG$, which suffices to prove \ref{spec} since the `gluing word' $u'$ does not depend on the choice of $w,w'\in \GGG$.

% the LAS property.  We prove it with $\tau = |u|$: given any $w^{(1)},w^{(2)}\in \GGG$, 
%by showing that
%To prove \ref{spec}, it suffices to show that 
%for any $w^{(1)},w^{(2)} \in \mathcal{G}$, there exists $u'$ with $|u'| = |u|$ for which $w^{(1)} u' w^{(2)} \in \mathcal{G}$. Given any such $w^{(1)},w^{(2)}$, by definition we can write $w^{(1)} = vt^{(1)}$, $w^{(2)} = vt^{(2)}$, and $vt^{(1)} u, vt^{(2)} u \in \mathcal{L}(X)$. Then, by LAS and the definition of $i$, there exists $v' \bar{t}^{(1)} u' v t^{(2)} u \in \mathcal{L}(X)$ where $v' \bar{t}^{(1)} u' \in B_i(v t^{(1)} u)$ and $v', \bar{t}^{(1)}, u'$ have the same lengths as $v$, $t^{(1)}$, $u$. However, since $u' v \in \mathcal{L}(X)$, by \eqref{eqn:i} we have $d_H(u', u) \geq i$. This implies that $v' = v$ and $\bar{t}^{(1)} = t^{(1)}$. We then have that $v t^{(1)} u' v t^{(2)} u \in \mathcal{L}(X)$, implying that $v t^{(1)} u' v t^{(2)} = w^{(1)} u' w^{(2)} \in \mathcal{G}$ and proving \ref{spec}.

Now we prove \ref{gap} by showing that $h(\mathcal{C}) < h(X)$. First, we note that by Theorem~\ref{LASentmin}, $X$ is entropy minimal.
If we define $X_v$ to be the subshift of points of $X$ containing no $v$ as a subword, then by entropy minimality, $h(X_v) < h(X)$. Since $\mathcal{C}^p$ is clearly factorial, $h(\mathcal{C}^p) = h(X_v)$ by Lemma~\ref{entbigger}. Therefore,
$h(\mathcal{C}^p) < h(X)$. To treat $\mathcal{C}^s$, note that if we define
\[
\mathcal{C}' = \{w \in \mathcal{L}(X) \ : \ w \textrm{ does not contain $u$ as a subword}\},
\]
then removing $u$ from the beginning of words in $\mathcal{C}^s$ places $(\mathcal{C}^s)_n$ in bijective correspondence with $(\mathcal{C}')_{n - |u|}$ for $n > |u|$. Then clearly $h(\mathcal{C}^s) = h(\mathcal{C}')$. If we define $X_u$ to be the subshift of points of $X$ containing no $u$ as a subword, then as above, by entropy minimality $h(X_u) < h(X)$. Since $\mathcal{C}'$ is clearly factorial, Lemma~\ref{entbigger} implies that $h(\mathcal{C}') \leq h(X_u)$, and so $h(\mathcal{C}^s) < h(X)$. It remains only to consider $\mathcal{B}$. If $w$ contains disjoint occurrences of $v$ and $u$ where the $v$ occurs to the left of the $u$, then we could write $w = xvyuz$ where $x$ contains no $v$ and $z$ contains no $u$. Then $vy \in \mathcal{G}$, $x \in \mathcal{C}^{p}$, and $uz \in \mathcal{C}^{s}$, meaning that $w \notin \mathcal{B}$. We've then shown that words in $\mathcal{B}$ either contain no $u$ at all, or can be written as $x u y$ where $x$ contains no $v$ and $y$ contains no $u$. This clearly implies that
\begin{equation}\label{Bbound}
|\mathcal{B}_n| \leq |(\mathcal{C}')_n| + \sum_{i = 0}^{n - |u|} |(\mathcal{C}^p)_i| \cdot |(\mathcal{C}')_{n - i}|.
\end{equation}

%\vc{This computation is simpler if you use a constant and let $n$ be arbitrary, instead of insisting that $K=1$ and then taking $n$ large, as in the previous version}
Choose $\epsilon > 0$ so that $h(\mathcal{C}^p), h(\mathcal{C}') < h(X) - 2\epsilon$. 
Then there is $K$ such that $|\CCC^p_n|, |\CCC'_n| \leq K e^{n(h(X)-\epsilon)}$ for every $n\in \NN$.
Combining this with (\ref{Bbound}) yields
\begin{multline*}
|\BBB_n| \leq K e^{n(h(X)-\epsilon)} + \sum_{i=0}^{n-|u|} K e^{i(h(X)-\epsilon)} K e^{(n-i)(h(X)-\epsilon)} \\
= K e^{n(h(X)-\epsilon)} + \sum_{i=0}^{n-|u|} K^2 e^{n(h(X)-\epsilon)}
\leq (K + nK^2) e^{n(h(X)-\epsilon)}.
\end{multline*}

Taking logs, dividing by $n$, and letting $n \rightarrow \infty$ yields 
$h(\mathcal{B}) \leq h(X) - \epsilon < h(X)$. Since $\mathcal{C} = \mathcal{C}^{p} \cup \mathcal{C}^{s} \cup \mathcal{B}$, we have shown \ref{gap}.

Finally, we prove \ref{inter} and \ref{union} with $L=|v|$.  For \ref{inter},
suppose $x\in X$ and $i\leq j\leq k\leq \ell$ are such that $x([i,k]), x([j,\ell]) \in \GGG$ and $k-j\geq L$.  Then $v$ is a prefix of $x([j,\ell])$; it is also a prefix of $x([j,k])$ since $k-j\geq L=|v|$.  Moreover, $x([i,k]) u \in \LLL$ and hence $x([j,k]) u \in \LLL$ as well, which shows that $x([j,k])\in \GGG$.
For \ref{union}, we see that $v$ is a prefix of $x([i,\ell])$ since it is a prefix of $x([i,k])$.  To show that $x([i,\ell])u\in \LLL$, we need to use the assumption that there is $a\leq i$ with $x([a,\ell]) \in \GGG$, which implies that $x([a,\ell])u\in \LLL$, and hence $x([i,\ell])u\in \LLL$.  Thus $x([i,\ell])\in \GGG$, proving \ref{union}.

We have verified the hypotheses of the first part of Theorem \ref{thm:towers}, which is enough to establish uniqueness and the central limit theorem.  The remaining conclusions of Theorem \ref{thm:las-main} will follow from the second half of Theorem \ref{thm:towers} once we prove that
\begin{equation}\label{eqn:gcd1}
\gcd\{|w| + |u| : w\in \GGG\} = 1.
\end{equation}
To see this, we produce
words $w,w' \in \GGG$ whose lengths differ by exactly $1$. First let $x\in \LLL$ be a word that starts and ends with $v$, and contains at least $2m+1$ disjoint occurrences of the word $v$; that is, $x=vz^{(1)}vz^{(2)}v \cdots vz^{(2m)} v$.  Note that such an $x$ exists by irreducibility.  Now let $a$ be a letter such that $y=xa\in \LLL$.  By LAS, there are $x' \in B_m(x)$ and $y'\in B_m(y)$ such that $x'u, y'u\in \LLL$.  Thus at least $m+1$ of the instances of $v$ in $x$ survive in $x'$, and similarly in $y'$.  In particular there is some instance of $v$ that survives in both $x'$ and $y'$; say this occurs as $x_{[i,i+|v|)}$. Then, we can take $w = v x'_{[i+|v|,|x|]} u \in \LLL$ and $w' = v y'_{[i+|v|,|x|+1]}u\in \LLL$; by definition both of these words are in $\GGG$, and their lengths clearly differ by $1$.  
This establishes \eqref{eqn:gcd1} and completes the proof of Theorem \ref{thm:las-main}, modulo the derivation in the next section of Theorem \ref{thm:towers} from the results in \cite{climenhaga1}.

\subsection{Proof of Theorem \ref{thm:towers}}\label{sec:derive}

Theorem \ref{thm:towers} is a consequence of results in \cite{climenhaga1}.  Since it is not stated in exactly this form there, we explain the necessary steps.  First note that \cite{climenhaga1} is given in terms of a potential function $\ph$, while we consider only MMEs, which correspond to the case $\ph=0$.  Also note that our condition \ref{spec} is denoted by [\textbf{I}$'$] in \cite[Theorem 3.2]{climenhaga1}.

The main result in that paper is \cite[Theorem 1.1]{climenhaga1}, which is given in terms of \ref{spec}, \ref{gap}, and a condition [\textbf{III}] that we do not state here, but which implies both \ref{inter} and \ref{union}.  It is shown there that if \ref{spec}, \ref{gap}, and [\textbf{III}] hold, then $X$ has a unique MME $\mu$, and that $\mu$ satisfies the stronger statistical properties listed in Theorem \ref{thm:las-main}, with two caveats:
\begin{itemize}
\item the result in \cite{climenhaga1} allows for a period $d$ in the Bernoulli property and exponential decay of correlations;
\item the definition of periodic orbit measures in \cite{climenhaga1} involves a sum over all orbits of length $\leq n$, instead of length exactly $n$.
\end{itemize}
Thus to prove Theorem \ref{thm:towers}, we must explain three things:
\begin{enumerate}
\item why \ref{inter} and \ref{union} are sufficient to replace [\textbf{III}];
\item why $d=1$ in our setting;
\item why our definition of periodic orbit measures still gives the equidistribution result.
\end{enumerate}
This requires a more careful examination of the results in \cite[\S3]{climenhaga1}, which establish \cite[Theorem 1.1]{climenhaga1}.  In \cite[Theorem 3.2]{climenhaga1}, it is shown that \ref{spec}, \ref{gap}, \ref{inter}, and \ref{union} guarantee existence of $\FFF \subset \LLL(X)$ such that the following are true.
\begin{itemize}
\item $\FFF$ is closed under concatenation: if $v,w\in \FFF$, then $vw\in \FFF$.  In particular, $\FFF$ satisfies \ref{spec} with $\tau=0$.
\item $\FFF$ satisfies \ref{gap}; there are $\mathcal{E}^p,\mathcal{E}^s \subset \LLL(X)$ such that $h(\mathcal{E}^p \cup \mathcal{E}^s \cup (\LLL(X) \setminus \mathcal{E}^p \FFF\mathcal{E}^s)) < h(X)$.
\item $\gcd \{|v| : v\in \FFF \} = \gcd \{|w| + \tau : w\in \GGG\}$.
\end{itemize}
In fact, \cite[Theorem 3.2]{climenhaga1} guarantees two further conditions [\textbf{II}$'$] and [\textbf{III}$^*$], which we do not state here but which are used in \cite[Theorem 3.1]{climenhaga1} to give the conclusions of \cite[Theorem 1.1]{climenhaga1} with $d=\gcd \{|v| : v\in \FFF \} = \gcd \{|w| + \tau : w\in \GGG\}$.  This demonstrates that under  the hypotheses of Theorem \ref{thm:towers}, we have $d=1$, and moreover that \ref{inter} and \ref{union} can replace [\textbf{III}] in \cite[Theorem 1.1]{climenhaga1}, which deals with the first two issues raised above.

It only remains to address the issue concerning periodic orbit measures.
By \cite[Lemma 4.5]{climenhaga1} applied to $\FFF$, there are $Q>0$ and $K,N\in \NN$ such that for $n \geq K$ there is $j\in [n-N,n]$ with $|\FFF_j| \geq Q e^{jh(X)}$.  
%Combining the argument above that $d=1$ with \cite[Remark 3.6]{climenhaga1}, we have $\gcd\{|w| : w\in \FFF \} = 1$.  
Since $\gcd\{|v| : v\in \FFF \} = 1$ and since $\FFF$ is closed under concatenation, there is $M\in \NN$ such that $\FFF_m\neq\emptyset$ for all $m\geq M$.  Given $n\geq K+M+N$, there is $j\in [n-M-N,n-M]$ such that
\[
|\FFF_j| \geq Q e^{jh(X)} \geq Q e^{(n-M-N)h(X)}.
\]
Now $\FFF_j \FFF_{n-j} \subset \FFF_n$ (by the fact that $\FFF$ is closed under concatenation) %since $\tau=0$ in \ref{spec}) 
and $\FFF_{n-j} \neq \emptyset$ (since $n-j \geq M$), so
\begin{equation}\label{eqn:Fn}
|\Per_n(X)| \geq |\FFF_n| \geq |\FFF_j| \geq Q e^{-(M+N)h(X)} e^{nh(X)},
\end{equation}
where the first inequality uses the fact that every word $w\in \FFF_n$ generates a periodic point of period $n$ by the free concatenation property.  Now let $\mu_n = \frac 1{|\Per_n(X)|} \sum_{x\in \Per_n(X)} \delta_x$ be the periodic orbit measures defined in \eqref{eqn:mun}, and let $\tilde\mu = \lim_k \mu_{n_k}$ be any weak* limit point of the sequence $\mu_n$.  Note that $\sigma_* \mu_n = \mu_n$, so this is a case of the construction from Lemma \ref{lem:build-mme}. By (\ref{eqn:Fn}), clearly 
$\liminf \frac 1{n_k} \ln |\Per_{n_k}(X)| = h(X)$, and we conclude that $h(\tilde\mu) = h(X)$.  Since $\mu$ is the unique MME, this implies that $\tilde\mu=\mu$.  This holds for every limit point $\tilde\mu$ of $\mu$, and therefore $\mu_n \to \mu$. %, which %completes the proof of Theorem \ref{thm:las-main}.

\section{Other proofs}\label{sec:other-pfs}

\subsection{Proof of Corollary \ref{thm:as-main}}\label{sec:corollary-pf}

The key observation for the proof of Corollary~\ref{thm:as-main} is that usual almost specification with $g \equiv 1$ implies either left or right almost specification.

\begin{lemma}\label{g=1}
If $X$ is irreducible and has almost specification with $g \equiv 1$, then either $X$ has left almost specification with $g \equiv 1$ or
$X$ has right almost specification with $g \equiv 1$. 
\end{lemma}

\begin{proof}
Suppose that $X$ is an irreducible subshift which has neither LAS with $g \equiv 1$ nor RAS with $g \equiv 1$. Then, there exist $u,v \in \mathcal{L}(X)$ for which $u'v \notin \mathcal{L}(X)$ for all $u' \in B_1(u)$, and there exist $U,V \in \mathcal{L}(X)$ for which $UV' \notin \mathcal{L}(X)$ for all $V' \in B_1(V)$. Use irreducibility to construct a word of the form $vwU \in \mathcal{L}(X)$.

We now claim that $X$ cannot have almost specification with $g \equiv 1$. Choose any $u' \in B_1(u)$ and $V' \in B_1(V)$. Then $u'v, UV' \notin \mathcal{L}(X)$. Then, if $y$ has the same length as $vwU$ and $u'yV' \in \mathcal{L}(X)$, both the initial $v$ and terminal $U$ from $vwU$ must have been changed in $y$. This means that $y \notin B_1(vwU)$, and so $X$ does not have almost specification with $g \equiv 1$.

By contrapositive, if $X$ is irreducible and has almost specification with $g \equiv 1$, it must have either LAS or RAS with $g \equiv 1$, completing the proof.
\end{proof}

We also need the following analogue of Theorem \ref{measure-center}.
%result, which we state without proof, as the proof is completely analogous to that of Theorem~\ref{measure-center}.

\begin{theorem}\label{measure-center-2}
If $X$ has almost specification with mistake function $g(n)$, then its measure center $\hat{X}$ is irreducible and has almost specification with mistake function $g(n)$.
\end{theorem}
\begin{proof}
Since $X$ has almost specification, it follows from \cite[Theorem 6.7]{WOC16} that $\hat{X}$ does as well.  Then irreducibility follows from \cite[Theorems 3.5, 3.7, and 4.3]{kwietniaketal}.  Note that the last of these requires an invariant measure whose support is $\hat{X}$; to build such a measure, enumerate the elements of $\LLL(\hat{X})$ as $v^{(1)},v^{(2)},\dots$, choose invariant measures $\mu_n$ with $\mu_n([v^{(n)}])>0$, and let $\mu = \sum_{n=1}^\infty 2^{-n} \mu_n$.

Because \cite[Theorem 6.7]{WOC16} does not explicitly state the mistake function for $\hat{X}$, we give a direct proof that $\hat{X}$ has almost specification with mistake function $g(n)$.
Given $w^{(1)},\dots,w^{(k)}\in \LLL(\hat{X})$, let $w^{(mk+r)}=w^{(r)}$ for all $m\in \NN$ and $r\in \{1,\dots, k\}$; for each $n$, we apply the almost specification property of $X$ to $w^{(1)},\dots,w^{(nk)}$ to obtain $x^{(n)} \in X$ such that $x^{(n)}_{[1,\infty)} = \bar{w}^{(n,1)}\dots\bar{w}^{(n,nk)}$ where $\bar{w}^{(n,i)} \in B_{g(|w^{(i)}|)}(w^{(i)})$ for each $1\leq i\leq nk$.  By compactness there is $n_j\to\infty$ such that $x=\lim_{j\to\infty} x^{(n_j)}\in X$ exists.  Moreover, we have $x_{[1,\infty)} = \mathbf{v}^{(1)} \mathbf{v}^{(2)} \cdots$ for some $\mathbf{v}^{(m)} = v^{(m,1)} v^{(m,2)} \cdots v^{(m,k)}$ with $v^{(m,r)}\in B_{g(|w^{(r)}|)}(w^{(r)})$.  Because there are only finitely many possibilities for each $\mathbf{v}^{(m)}$, there is $\mathbf{\bar v} = \bar{v}^{(1)} \cdots \bar{v}^{(k)}$ such that $M = \{m\in \NN : \mathbf{v}^{(m)} = \mathbf{\bar v}\}$ has positive upper density.  By the same argument as in the proof of Theorem \ref{measure-center}, there is an ergodic measure $\mu$ with $\mu([\mathbf{\bar v}])>0$, and we conclude that $\mathbf{\bar v}\in \LLL(\hat X)$, which completes the proof of Theorem \ref{measure-center-2}.
\end{proof}

Now we complete the proof of Corollary~\ref{thm:as-main}.
By Theorem~\ref{measure-center-2}, we may assume without loss of generality that $X$ is irreducible. Then, by Theorem~\ref{g=1}, $X$ either has LAS with $g \equiv 1$ or RAS with $g \equiv 1$. In the former case, the proof is complete by applying Theorem~\ref{thm:las-main}. 

If $X$ has RAS with $g \equiv 1$, then define a map $\rho\colon A^\zz \to A^\zz$ by 
\[
\rho( \ldots x(-1) x(0) x(1) \ldots ) = \ldots x(1) x(0) x(-1) \ldots.
\]
Then $\rho(X)$ is a subshift, and since $X$ was irreducible and had RAS with $g \equiv 1$, it's clear that $\rho(X)$ is irreducible and has LAS with $g \equiv 1$. Therefore, $\rho(X)$ has a unique measure of maximal entropy satisfying the conclusions of Theorem~\ref{thm:las-main}. Since $\rho$ is bijective on periodic points of $X$ and induces an entropy-preserving bijection between $\mathcal{M}(X)$ and $\mathcal{M}(\rho(X))$, there is a corresponding unique measure of maximal entropy on $X$ with the same properties.

\subsection{Proof of Theorem \ref{doublelogex}}\label{sec:loglog-pf}

Having shown that LAS with bounded $g$ implies intrinsic ergodicity, it is natural to wonder whether this result can be strengthened, i.e.\ whether other upper bounds on $g$ may also imply uniqueness. We can now prove Theorem~\ref{doublelogex}, which shows that multiple MMEs can occur within subshifts with LAS for $g$ on the order of $\log \log n$. We will use the following technical lemma from \cite{pavlovAS}.

\begin{lemma}{\rm (\cite{pavlovAS}, Lemma 4.4)}\label{2span}
For every alphabet $A$ and positive integer $n$, there exists a set $U_n \subseteq A^n$ such that 
$|U_n| \leq \frac{16}{n^2} |A|^n$ and $U_n$ is $2$-spanning with respect to the Hamming metric, i.e. $\forall w \in A^n$, $\exists u \in U_n$ with $d_H(w,u) \leq 2$.
\end{lemma}

We will actually exhibit a subshift with RAS rather than LAS; as argued in the proof of Corollary~\ref{thm:as-main}, this obviously can be turned into an LAS example by simply reflecting sequences about the origin. Our subshift will have alphabet
\[
A = \{-N, \ldots, -1, 1, \ldots, N\},
\]
where $N > 2^{17} + 4$. We begin by using Lemma~\ref{2span} to fix, for every $n$, a set $U_n \subseteq \{1, \ldots, N\}^n$ which is $2$-spanning in $\{1, \ldots, N\}^n$ and with
$|U_n| \leq \frac{16}{n^2} N^n$.
We will use $U_n$ to define certain collections $T^+ \subset \{1,\dots, N\}^*$ and $T^- \subset \{-1,\dots, -N\}^*$, and then take $X$ to be the coded shift generated by $T := T^+ \cup T^-$, i.e.\ the closure of the set of all bi-infinite concatenations of words from $T$.

%Now, we define some slightly more complicated sets. 
For every $n \geq 4$, define $k = k(n) = \lfloor \log_2 \log_2 n \rfloor$ (note that $2^{2^k} \leq n < 2^{2^{k+1}}$) and define $T^+_n$ to be the set
\[
\{w \in \{1, \ldots, N\}^n \ : \ w(1) = 1, \forall i \in [0,k-1], w([2^{2^i} + 1, 2^{2^{i+1}}]) \in U_{2^{2^{i+1}} - 2^{2^i}}\}.
\]
Define $T^+ = \bigcup_n T^+_n$. We first note that the collection $T^+$ is prefix-closed, meaning that $w \in T^+$, 
$w = uv \Longrightarrow u \in T^+$; this follows immediately from the definition of the sets $T^+_n$. Also, for every $n$, $T^+_n$ is
$(1 + 2\lfloor \log_2 \log_2 n \rfloor)$-spanning in $\{1, \ldots, N\}^n$; given a word $w \in \{1, \ldots, N\}^n$, one only needs to change 
the first letter to $1$ and make at most two changes to each $w([2^{2^i} + 1, 2^{2^{i+1}}])$, $i \in [0,k-1]$, to place it in $U_{2^{2^{i+1}} - 2^{2^i}}$.

Finally, for every $n$, $|T_n^+|$ is bounded above by
\[
%|T^+_n| \leq 
N^{n-1} \cdot \prod_{i=0}^{k-1} \frac{16}{(2^{2^{i+1}} - 2^{2^i})^2} \leq 
N^{n-1} \cdot \prod_{i=0}^{k-1} \frac{16}{(2^{2^{i+1} - 1})^2} = N^n \left(\frac{1}{N \cdot 2^{2^{k+2} - 6k - 4}}\right).
\]
Then for all $n$, 
\begin{equation}\label{examplebound0}
|T^+_n| \leq \frac{N^n}{N}.
\end{equation}

When $n \geq 2^{16}$, then $k \geq 4$, so $2^{k+2} - 6k - 4 > 1.125 \cdot 2^{k+1}$. This means that for $n \geq 2^{16}$, we get the stronger bound
\begin{equation}\label{examplebound}
|T^+_n| \leq \frac{N^n}{N \cdot n^{1.125}}.
\end{equation}

Now, for every $n$, define $T^-_n \subseteq \{-N, \ldots, -1\}^n$ by $T^-_n = -T^+_n = \{-w \ : \ w \in T^+_n\}$, and define
$T^- = \bigcup_n T^-_n$. Clearly $T^-$ is prefix-closed since $T^+$ is. Similarly, $T^-_n$ is 
$(1 + 2\lfloor \log_2 \log_2 n \rfloor)$-spanning in $\{-N, \ldots, -1\}^n$ and has the same cardinality bound as 
that of (\ref{examplebound}). Define, for every $n$, $T_n = T^-_n \cup T^+_n$; clearly $T_n$ is $(1 + 2\lfloor \log_2 \log_2 n \rfloor)$-spanning within the set of $n$-letter words with constant sign. Then, define $T = T^+ \cup T^- = \bigcup_n T_n$. We are finally prepared to define our subshift $X$; it is simply the coded system defined by $T$.

We first note that $\{1, \ldots, N\}^{\mathbb{Z}} \subset X$. To see this, note that for every $k$, the final 
$2^{2^{k+1}} - 2^{2^k} - 1 > k$ letters of words in $T^+_{2^{2^{k+1}} - 1}$ are completely unconstrained. In other words, every word in 
$\{1, \ldots, N\}^*$ is the suffix of some word in $T^+$, and so by taking limits as $k \rightarrow \infty$, we see that $\{1, \ldots, N\}^{\mathbb{Z}} \subset X$. Similarly, every word in $\{-N, \ldots, -1\}^*$ is the suffix of some word in $T^-$, and so 
$\{-N, \ldots, -1\}^{\mathbb{Z}} \subset X$. Trivially, this implies that $h(X) \geq \log N$. We will now bound $|\LL_n(X)|$ from above similarly to \cite{pavlovAS} to show that in fact $h(X) = \log N$.

For every $n$ and every $w \in \LL_n(X)$, we can decompose $w$ as $w = w^{(1)} \ldots w^{(k)}$, where $w^{(1)}$ is the suffix of a word in $T$, 
$w^{(2)}, \ldots, w^{(k-1)} \in T$, and $w^{(k)}$ is the prefix of a word in $T$. Recall that $T$ is prefix-closed, and that every word of constant sign is the suffix of some word in $T$. Therefore, we can rephrase by saying that $w^{(1)}$ has constant sign, and that 
$w^{(2)}, \ldots, w^{(k)} \in T$. Defining 
$n_i = |w^{(i)}|$ for $1 \leq i \leq k$, we see that
\begin{equation}\label{examplebound2}
\begin{aligned}
|\LL_n(X)| &\leq \sum_{k=1}^{n} \sum_{\sum_{i=1}^k n_i = n} 2N^{n_1} \prod_{i=2}^{k} |T_{n_i}|\\
&\leq 2N^n \sum_{k=1}^{n} \sum_{\sum_{i=1}^k n_i = n} \prod_{i=2}^{k} \frac{|T_{n_i}|}{N^{n_i}}
\leq 2N^n \sum_{k=1}^n \left( \sum_{j=1}^\infty \frac{|T_j|}{N^j} \right)^{k-1}.
\end{aligned}
\end{equation}
By (\ref{examplebound0}) and (\ref{examplebound}), we have
\[
\sum_{j=1}^\infty \frac{|T_j|}{N^j} \leq \sum_{j=1}^{2^{16}} \frac{2}{N} + \sum_{2^{16}+1}^{\infty} \frac{2}{N \cdot j^{1.125}} < \frac{2^{17}}{N} + \int_{2^{16}}^{\infty} \frac{2}{N \cdot x^{1.125}} \ dx = \frac{2^{17} + 4}{N}.
\]

Since $N > 2^{17} + 4$, we can define $\alpha := \frac{2^{17} + 4}{N} < 1$, and then (\ref{examplebound2}) shows that
$|\LL_n(X)| < 2N^n \frac{1}{1 - \alpha}$. Taking logs, dividing by $n$, and letting $n \rightarrow \infty$ shows that $h(X) \leq \log N$,
and so $h(X) = \log N$. This immediately implies that $X$ has at least two measures of maximal entropy (with disjoint supports), namely the uniform Bernoulli measures on the full shifts $\{1, \ldots, N\}^{\mathbb{Z}}$ and $\{-N, \ldots, -1\}^{\mathbb{Z}}$, both of which are contained in $X$.

It only remains to show that $X$ has RAS with $g(n) = 1 + \lfloor 2 \log_2 \log_2 n \rfloor$. To see this, consider any $v,w \in \LL(X)$. Decompose each into maximal words of constant sign as $v = v^{(1)} v^{(2)} \ldots v^{(m)}$ and $w = w^{(1)} w^{(2)} \ldots w^{(n)}$, i.e. every $v^{(j)}$ and $w^{(j)}$ has constant sign, which alternates as $j$ varies. Then, $v^{(1)}$ is a suffix of a word $t$ in $T$, and all $v^{(j)}$ for $j > 1$ are in $T$ (recall that $T$ is prefix-closed). Similarly, $w^{(j)}$ is in $T$ for $j > 1$. Let $\ell = |w^{(1)}|$. Since
$T_{\ell}$ is $(1 + 2\lfloor \log_2 \log_2 \ell \rfloor)$-spanning in $\{-N, \ldots, -1\}^{\ell} \cup \{1, \ldots, N\}^{\ell}$
and $w^{(1)}$ has constant sign, there exists $u^{(1)} \in B_{1 + 2\lfloor \log_2 \log_2 \ell \rfloor}(w^{(1)})$ which is also in $T$. 
Since it is a concatenation of words in $T$, $t v^{(2)} \ldots v^{(m)} u^{(1)} w^{(2)} \ldots w^{(n)} \in \mathcal{L}(X)$. 
Then its subword $v^{(1)} \ldots v^{(m)} u^{(1)} w^{(2)} \ldots w^{(n)}$ must also be in $\mathcal{L}(X)$.
Finally, since $d_H(u^{(1)} w^{(2)} \ldots w^{(m)}, w) = d_H(u^{(1)}, w^{(1)}) \leq 1 + 2 \lfloor \log_2 \log_2 \ell \rfloor \leq 
1 + 2 \log_2 \log_2 |w|$, we have shown that $X$ has RAS with $g(n) = 1 + 2 \lfloor \log_2 \log_2 n \rfloor$.

\subsection{Proof of Proposition \ref{prodfact}}\label{sec:factors-pf}

Now we prove that the property of having LAS with bounded $g$ is preserved under products and factors.

Suppose that $X, Y$, with alphabets $A,B$ respectively, have LAS with $g \equiv m$ and $g \equiv n$ respectively. Then $X \times Y$ is a subshift with
alphabet $A \times B$. Consider any words $w^{(1)}, w^{(2)} \in \mathcal{L}(X \times Y)$. Let $u^{(1)} \in A^*$ and $v^{(1)} \in B^*$ denote the first and second coordinates of $w^{(1)}$ respectively; by definition $u^{(1)} \in \mathcal{L}(X)$ and $v^{(1)} \in \mathcal{L}(Y)$. Similarly define $u^{(2)} \in \mathcal{L}(X)$ and $v^{(2)} \in \mathcal{L}(Y)$ via 
$w^{(2)}$. Then, by LAS of $X$, there exist $\bar{u}^{(1)} \in B_m(u^{(1)})$, $\bar{v}^{(1)} \in B_n(v^{(1)})$ so that $\bar{u}^{(1)} u^{(2)} \in \mathcal{L}(X)$ and $\bar{v}^{(1)} v^{(2)} \in \mathcal{L}(Y)$. 
Then, if we define $\bar{w}^{(1)}$ to have first and second coordinates $\bar{u}^{(1)}$ and $\bar{v}^{(1)}$, we get $\bar{w}^{(1)} w^{(2)} \in \mathcal{L}(X \times Y)$, and 
$\bar{w}^{(1)} \in B_{m + n}(w^{(1)})$. Therefore, $X \times Y$ has LAS with $g \equiv m + n$.

Now, consider a factor map $\phi$ on $X$. Then $\phi$ is a sliding block code with some radius $r$, i.e. $x([i-r,i+r])$ uniquely 
determines $(\phi(x))(i)$ for every $x \in X$, $i \in \mathbb{Z}$. Take any words $w^{(1)} \in \LL_{n_1}(\phi(X))$ and 
$w^{(2)} \in \LL_{n_2}(\phi(X))$. There then exist words $v^{(1)} \in \LL_{n_1 + 2r}(X)$, $v^{(2)} \in \LL_{n_2 + 2r}(X)$ so that
$\phi(v^{(1)}) = w^{(1)}$ and $\phi(v^{(2)}) = w^{(2)}$. Create $u^{(1)} \in \LL_{n_1}(X)$ by removing the final $2r$ letters of $v^{(1)}$. 
Then by definition of sliding block code, $p^{(1)} := \phi(u^{(1)})$ is a prefix of $w^{(1)}$ of length $n_1 - 2r$.

Then by LAS of $X$, there exists $\bar{u}^{(1)} \in B_m(u^{(1)})$ for which $\bar{u}^{(1)} v^{(2)} \in \LL_{n_1 + n_2 + 2r}(X)$. Define
$y = \phi(\bar{u}^{(1)} v^{(2)})$; then $y \in \LL_{n_1 + n_2}(\phi(X))$. Since $v^{(2)}$ is a suffix of $\bar{u}^{(1)} v^{(2)}$, and $\phi(v^{(2)}) = w^{(2)}$, 
$y = \phi(\bar{u}^{(1)} v^{(2)})$ has $w^{(2)}$ as a suffix. Also, since $\bar{u}^{(1)} \in B_m(u^{(1)})$, and $\phi$ has radius $r$,
$\bar{p}^{(1)} := \phi(\bar{u}^{(1)}) \in B_{m(2r+1)}(\phi(u^{(1)})) = B_{m(2r+1)}(p^{(1)})$. (The only differences in $\phi(\bar{u}^{(1)})$ and $\phi(u^{(1)})$ must be at locations within distance $r$ from a difference within $\bar{u}^{(1)}$ and $u^{(1)}$.) Since $\bar{u}^{(1)}$ is a prefix of $\bar{u}^{(1)} v^{(2)}$, 
$y = \phi(\bar{u}^{(1)} v^{(2)})$ has $\bar{p}^{(1)}$ as a prefix. But then we can write $y = \bar{p}^{(1)} x w^{(2)}$ for some $x \in \LL_{2r}(\phi(X))$. Since 
$\bar{p}^{(1)} \in B_{m(2r+1)}(p^{(1)})$, $\bar{p}^{(1)} x \in B_{2r + m(2r+1)}(w^{(1)})$. Since $y \in \mathcal{L}(\phi(X))$, we've proved that $\phi(X)$
has LAS with $g \equiv 2r + m(2r+1)$. 

\subsection{Proof of Lemma \ref{entbigger}}\label{sec:entropy-pf}

We will show that for a factorial set $\mathcal{D}$ of words, the entropy of $\mathcal{D}$ is the same as that of the subshift ``generated by'' $\mathcal{D}$.

%\vc{I rewrote this proof to simplify some of the notation; it's the same proof you wrote down, but because $\DDD$ is factorial the limit in $h(\DDD)$ exists and in particular we don't need to pass to a subsequence, which simplifies the presentation.}
Given a word $w\in A^*$, a sufficient condition to have $w\in \LLL(X(\DDD))$ is that there are infinitely many $k\in \NN$ and $u,v\in A^k$ such that $uwv\in \DDD$.  Given $i,j\in \NN$, consider the set
\[
\DDD_i^{(j)} = \{w\in \DDD_i : \text{there are } u,v\in \DDD_j \text{ such that } uwv\in \DDD\}.
\]
Note that by definition, for every $n \in \mathbb{N}$,
\[
\DDD_n = \DDD_n^{(0)} \supset \DDD_n^{(n)} \supset \DDD_n^{(2n)} \supset \cdots
\]
Since these sets are finite, this sequence stabilizes, i.e. there exists $K \in \mathbb{N}$ so that for all $k \geq K$, 
\begin{equation}\label{referee}
\DDD_n^{(kn)} = \bigcap_{j \in \mathbb{N}} \DDD_n^{(jn)} \subseteq \mathcal{L}_n(X(\mathcal{D})).
\end{equation}

Our goal is to prove that for any $n\in \NN$ and $\epsilon>0$, we have
\begin{equation}\label{eqn:Dnkn}
|\DDD_n^{(kn)}| \geq e^{n(h(\DDD) - 2\epsilon)}
\end{equation}
for all sufficiently large $k$; using this we will produce `enough' words in $\LLL_n(X(\DDD))$.

%We start by proving \eqref{eqn:Dnkn}.
Since $\DDD$ is factorial, Lemma \ref{wordcount} gives $|\DDD_n| \geq e^{nh(\DDD)}$ for all $n$.  Fix $\epsilon>0$, and let $N = N(\epsilon)$ be such that
\begin{equation}\label{eqn:Dn}
e^{nh(\DDD)} \leq |\DDD_n| \leq e^{n(h(\DDD)+\epsilon)}
\end{equation}
for all $n\geq N$.  For any $k,n\in \NN$ we observe that if $w\in \DDD_{kn}^{(kn)}$, then
\[
w([1,n]), w([n+1,2n]), \dots, w([(k-1)n+1,kn]) \in \DDD_n^{(kn)}
\]
by factoriality of $\DDD$.  In particular, this gives
\begin{equation}\label{eqn:Dknkn}
|\DDD_{kn}^{(kn)}| \leq |\DDD_n^{(kn)}|^k.
\end{equation}
On the other hand, every $w\in \DDD_{3kn}$ has $w([kn+1,2kn]) \in \DDD_{kn}^{(kn)}$, and so
\[
e^{3knh(\DDD)}\leq
|\DDD_{3kn}| \leq |\DDD_{kn}|^2 |\DDD_{kn}^{(kn)}|
\leq e^{2kn(h(\DDD)+\epsilon)} |\DDD_n^{(kn)}|^k
\]
for every $k,n$ with $kn > N$, where the last inequality uses \eqref{eqn:Dn} and \eqref{eqn:Dknkn}.  Dividing by $e^{2kn(h(\DDD)+\epsilon)}$ and taking a $k$th root gives \eqref{eqn:Dnkn} for all $k > N/n$.

Now, by (\ref{referee}) and (\ref{eqn:Dnkn}),
\[
|\mathcal{L}_n(X(\mathcal{D}))| \geq \Big| \bigcap_{j \in \mathbb{N}} \DDD_n^{(jn)} \Big| \geq e^{n(h(\mathcal{D}) - 2\epsilon)}.
\]
Dividing by $n$, taking logs, and letting $n \rightarrow \infty$ yields $h(X(\DDD)) \geq h(\DDD) - 2\epsilon$.  Since $\epsilon>0$ was arbitrary, we have $h(X(\DDD)) \geq h(\DDD)$, which completes the proof of Lemma \ref{entbigger}.

\subsection{Unbounded mistake functions}\label{sec:unbdd}
%{Proof of Theorem~\ref{notentmin}}\label{sec:notentmin-pf}

As mentioned in Remark \ref{rmk:unbdd}, Theorem \ref{LASentmin} fails as soon as $g$ is unbounded.  To demonstrate this, we give a family of irreducible examples with LAS for arbitrary unbounded $g$ which are not entropy minimal, and in fact even have zero entropy.
%We will show that no matter how slowly an unbounded mistake function $g$ grows, LAS with mistake function $g$ does not imply entropy minimality.

Consider any unbounded $g\colon \mathbb{N} \rightarrow \mathbb{N}$ where $\frac{g(n)}{n} \rightarrow 0$ as $n \rightarrow \infty$. We may assume without loss of generality that $g$ is nondecreasing since any subshift with LAS with mistake function $g(n)$ also has LAS for any larger mistake function. Then, define $X$ to have alphabet $\{0,1\}$ and consist of all $x \in \{0,1\}^\zz$ such that for every $n$ and every $n$-letter subword $w$ of $x$, the number of $1$ symbols in $w$ is less than or equal to $g(n)$. (These shifts were defined in \cite{stanley} under the name \textbf{bounded density shifts.}) We note that $\mathcal{L}(X)$ is exactly the set of words $v$ where for every $n$ and every subword $w$ of length $n$, the number of $1$s in $w$ is less than or equal to $g(n)$; any such word $v$ is in $\mathcal{L}(X)$ since $0^{\infty} v 0^{\infty} \in X$.

We first show that $X$ has LAS with mistake function $g(n)$. For any $v,w \in \mathcal{L}(X)$, the number of $1$s in $v$ is at most 
$g(|v|)$, and so $v' = 0^{|v|} \in B_{g(|v|)}(v)$. Then, we claim that $v'w \in \mathcal{L}(X)$. To see this, note that for every subword $u$ of $v'w$, the number of $1$s in $u$ is equal to the number of $1$s in $u'$, the maximal subword of $u$ contained in $w$. Then since $w \in \mathcal{L}(X)$, this is less than or equal to $g(|u'|) \leq g(|u|)$, and so $v'w \in \mathcal{L}(X)$, completing the proof of LAS.

Now, we show that $X$ is irreducible. Consider any $v,w \in \mathcal{L}(X)$, and let $n \geq \max\{|v|, |w|\}$.  and without loss of generality assume that they have the same length $n$. Since $g$ is unbounded, there exists $N$ so that $g(N) \geq 2n$. Then, let
$u = v 0^N w$. Any subword $u'$ of $u$ which does not contain letters of both $v$ and $w$ is a subword of either $v 0^N$ or $0^N w$, and so has number of $1$s less than or equal to $g(|u'|)$ exactly as argued above. Any other subword $u''$ of $u$ contains letters of both $v$ and $w$, and so has length greater than $N$. The number of $1$s in $u''$ is then less than or equal to $2n \leq g(N) \leq g(|u''|)$. We've then shown that $u \in \mathcal{L}(X)$, proving irreducibility of $X$.

We also claim that $h(X) = 0$. To see this, note that since $\frac{g(n)}{n} \rightarrow 0$, the limiting frequency of $1$ symbols in every point of $x$ is zero. Therefore, an invariant measure $\mu$ on $X$ must have $\mu([1]) = 0$, and so the only such measure is the delta measure on the sequence of all $0$s. This measure clearly has entropy $0$, and so by the variational principle, $h(X) = 0$. Then, $X$ is not entropy minimal, since it contains the proper subshift containing only the point of all $0$s, which also has entropy $0$.

%\section{Open problems}\label{sec:open}

%\vc{I'm not convinced we need a separated section for this, but it seemed like the most natural place to put Theorem \ref{notentmin}. I added a little more discussion but we could remove it, I'm not wedded to it}

\bibliographystyle{plain}
\bibliography{leftASv2}

\end{document}